\documentclass[10pt, a4paper]{amsart}
\usepackage{amsthm}
\usepackage[dvips]{graphicx}
{\theoremstyle{definition}
\newtheorem{dfn}{Definition}}
\newtheorem{prop}[dfn]{Proposition}

\newtheorem{thm}[dfn]{Theorem}
{\theoremstyle{definition}
\newtheorem{rem}[dfn]{Remark}}
\newtheorem{lem}[dfn]{Lemma}

{\theoremstyle{definition}
\newtheorem{exa}[dfn]{Example}}

\usepackage[top=20truemm,bottom=20truemm,left=25truemm,right=25truemm]{geometry}
\usepackage{amsmath,amssymb,amscd,bm,graphicx,tikz-cd,upgreek,dsfont,amsfonts,tikz,enumitem,euscript}

\usepackage[colorlinks=true,
    citecolor=denim,
    linkcolor=alizarin,
    urlcolor=lightseagreen]{hyperref}
    
\definecolor{alizarin}{rgb}{0.82, 0.1, 0.26}
\definecolor{azure(colorwheel)}{rgb}{0.0, 0.5, 1.0}
\definecolor{blue(pigment)}{rgb}{0.2, 0.2, 0.6}
\definecolor{denim}{rgb}{0.08, 0.38, 0.74}
\definecolor{mint}{rgb}{0.24, 0.71, 0.54}
\definecolor{parisgreen}{rgb}{0.31, 0.78, 0.47}
\definecolor{persiangreen}{rgb}{0.0, 0.65, 0.58}
\definecolor{seagreen}{rgb}{0.18, 0.55, 0.34}
\definecolor{shamrockgreen}{rgb}{0.0, 0.62, 0.38}
\definecolor{green(pigment)}{rgb}{0.0, 0.65, 0.31}
\definecolor{cadmiumgreen}{rgb}{0.0, 0.42, 0.24}
\definecolor{lightseagreen}{rgb}{0.13, 0.7, 0.67}
\definecolor{mediumseagreen}{rgb}{0.24, 0.7, 0.44}
\definecolor{pinegreen}{rgb}{0.0, 0.47, 0.44}
\definecolor{tealgreen}{rgb}{0.0, 0.51, 0.5}

\usepackage{color}
\usepackage{comment}

\newcommand{\bC}{\mathbb{C}}
\newcommand{\bF}{\mathbb{F}}
\newcommand{\bH}{\mathbb{H}}

\newcommand{\bQ}{\mathbb{Q}}

\newcommand{\bZ}{\mathbb{Z}}

\begin{document}

\title[Twisted Alexander polynomials of tunnel number one Montesinos knots]{Twisted Alexander polynomials of tunnel number one Montesinos knots} 

\author[A. Aso]{Airi Aso}
\address{Depertment of Mathematical sciences
  Graduate School of Science 
  Tokyo Metropolitan University
  1-1 Minamiohsawa, Hachioji-shi, Tokyo, 192-0397 JAPAN} 
\email{aso-airi@ed.tmu.ac.jp}

\begin{abstract}
 We calculate the twisted Alexander polynomials of all tunnel number one Montesinos knots associated to their $SL_2(\bC)$-representations and obtain their leading coefficients and degrees.
 As a corollary, we get some interesting examples, that is,  nonfibered knots with monic Alexander polynomials which have non-monic twisted Alexander polynomials.
\end{abstract}

\maketitle

\section{Introduction}

The twisted Alexander polynomial is a generalization of the Alexander polynomial, and it is defined for the pair of a finitely presented group and its representations. 
The notion of twisted Alexander polynomials was introduced by Wada \cite{W} and Lin \cite{L} independently in 1990s. 
The definition of Lin is for knots in $S^3$ and the definition of Wada is for finitely presented groups. 
In this paper, we use Wada's twisted Alexander polynomials defined by the following;

\begin{dfn}\label{def of TAP}
Let $K$ be a knot in $S^3$ and
\[
G(K) = \langle x_1, \ldots ,x_n \ \vline \  r_1, \ldots , r_{n-1}  \rangle
\]
be the knot group of $K$, that is, the fundamental group of the knot exterior. 
Then, the {\it twisted Alexander polynomial} of $K$ associated to a representation $\rho: G(K) \to SL_n (\bF)$ is given by
\[
\Delta_{K,\rho} (t)=\displaystyle \frac{\det A_{\rho,k}}{\det [(\rho \otimes \frak{a}) \circ \phi(x_k-1)]},
\]
where $\frak{a}:\bZ G(K) \to \bZ [ t,t^{-1} ]$ is the abelianization of the group ring $\bZ G(K)$, 
$\phi: \bZ \Gamma \to \bZ G(K)$ is the natural ring homomorphism of the free group $\Gamma$  generated by $x_1, \cdots ,x_n $ and $A_{\rho,k}$ is defined for any $1 \le k \le n$ as follows:

$A_{\rho,k}$ is the $d (n-1)\times d (n-1)$ matrix defined by 
\[
\left(
\begin{array}{cccccc}
A_{1,1}  &\cdots & A_{1,k-1} &  A_{1,k+1} & \cdots & A_{1,n} \\
  \vdots & & \vdots & \vdots & & \vdots\\
A_{n-1,1}  & \cdots & A_{n-1,k-1} & A_{n-1,k+1} & \cdots & A_{n-1,n} \\
\end{array}
\right),
\]
where
\[
A_{i,j} = (\rho \otimes \frak{a}) \circ \phi \left( \frac{\partial r_i}{\partial x_j}\right).
\]
Here $\displaystyle \frac{\partial}{\partial x_j} :\bZ \Gamma \to \bZ \Gamma$ denotes the Fox derivative with respect to $x_j$.
\end{dfn}

By Kitano and Morifuji \cite{KM}, it is proved that Wada's twisted Alexander polynomials of the knot groups for any nonabelian representations into $SL_2(\bF)$ over a field $\bF$ are polynomials. 
As a corollary of the claim, they also showed that if $K$ is a fibered knot of genus $g$, then its twisted Alexander polynomials are monic and has degree $4g-2$ for any nonabelian $SL_2(\bF)$-representations. 
It is also showed that there exists a nonfibered knot which has an $SL_2(\bC)$-representation such that the twisted Alexander polynomial of the representation is monic (see \cite{GoMo}).
Dunfield, Friedl and Jackson \cite{DFJ} conjectured that the twisted Alexander polynomials of hyperbolic knots
associated to their holonomy representations into $SL_2(\bC)$ 
determines the genus and fiberedness of the knots.
In fact, they verified the conjecture for all hyperbolic knots up to 15 crossings.

\begin{figure}[h]
  \begin{center}
\includegraphics[clip,width=10cm]{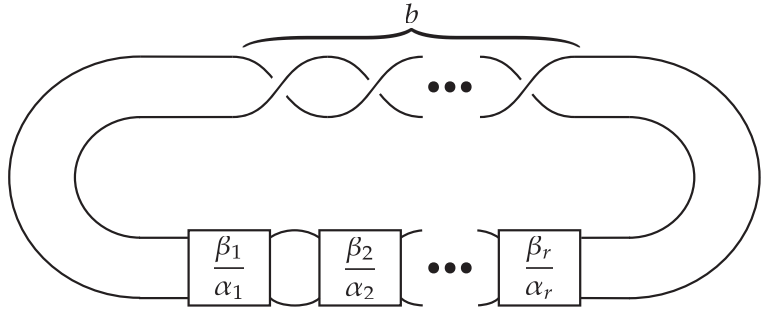}
\caption{A Montesinos knot $M(b;(\alpha_1, \beta_1), (\alpha_2, \beta_2), \cdots ,(\alpha_r, \beta_r))$}
  \end{center}
\end{figure}

A Montesinos knot 
\[
M(b;(\alpha_1, \beta_1), (\alpha_2, \beta_2), \cdots ,(\alpha_r, \beta_r))
\]
is the knot depicted in Figure 1
where the box $\beta_i/\alpha_i$ (with $\gcd(\alpha_i, \beta_i)=1$ for each $i$) represents a rational tangle of the slope $\beta_i/\alpha_i$ and the integer $b$ denotes the number of half twists (if $b$ is negative, each closing of the twists are opposite).
Then, each rational numbers $\beta/\alpha$ has continued fraction expansions, i.e. 
\[
\frac{\beta}{\alpha} =
c_0 + \frac{1}{c_1+\displaystyle \frac{1}{c_2 + \displaystyle \frac{1}{ \ddots +\displaystyle \frac{1}{c_k}}}}
=: [c_0, c_1, \ldots ,c_k],
\]
where $c_0, c_1, \ldots ,c_k$ are integers.
Each integer $c_i$ corresponds to the number of twists depicted in Figure 2.
A rational number $\beta/\alpha$ has some continued fraction expansions, however, they has a continued fraction expansion $[2 c_0, 2 c_1, \ldots , 2c_k]$ where either $\alpha$ or $\beta$ is even (if $\alpha$ is even, then $k$ is odd, and if $\beta$ is odd, then $k$ is even).
The class of Montesinos knots is one of the special classes of knots in $S^3$.
It contains many important families of knots, for example, two-bridge knots and pretzel knots.
The genus of all Montesinos knots are completely determined by Hirasawa and Murasugi \cite{HM}.

\begin{figure}[h]
  \begin{center}
\includegraphics[clip,width=11cm]{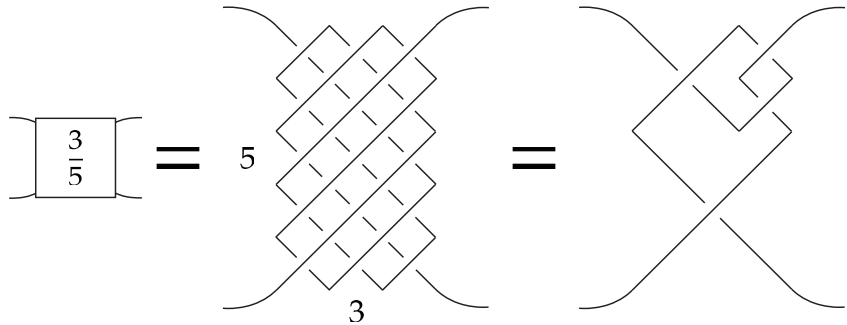}
\caption{A continued fractional expansion of tangle $\frac{3}{5}=\frac{1}{1 + \frac{1}{1+\frac{1}{2}}}$}
  \end{center}
\end{figure}

Let $K$ be a knot in $S^3$.
Then, the tunnel number $\tau(K)$ of $K$ is the minimal number of mutually disjoint arcs, say ${\tau_i}$ "properly embedded" in the pair $(S^3,K)$ such that the complement of an open regular neighborhood of $K \cup (\cup \tau_i)$ is a handle body.

In this paper, we compute the twisted Alexander polynomials of tunnel number one Montesinos knots.
To this end, we have the following theorem;

\begin{thm}[Klimenko--Sakuma \cite{KS}]\label{Klimenko--Sakuma}
A Montesinos knot $M(b;(\alpha_1,\beta_1), (\alpha_2, \beta_2), \cdots , (\alpha_r,\beta_r))$ has tunnel number one if and only if one of the following conditions holds up to cyclic permutation of the indices :\\
{\rm (1)}\ $r=2$\\
{\rm (2)}\ $r=3$, $\alpha_1=2$, and $\alpha_2 \equiv \alpha_3 \equiv 1 \mod 2$\\
{\rm (3)}\  $r=3$, $\beta_2/\alpha_2 \equiv \beta_3/\alpha_3 \equiv \pm1/3 \in \bQ/\bZ$, and $e(K)=b- \sum^r_{i=1} \beta_i/\alpha_i=\pm 1/(3 \alpha_1)$
\end{thm}

It is known that if $K$ is a fibered knot of genus $g$, then its Alexander polynomial $\Delta_{K}(t)$ is monic and $\deg(\Delta_{K}(t)) = 2 g$.
In \cite{HM}, it is shown that most of tunnel number one Montesinos knots are fibered if their Alexander polynomials are monic.
In this paper, as a corollary of the calculation of tunnel number one Montesinos knots, we consider the twisted Alexander polynomials of the exceptional knots i.e. nonfibered knots which has monic Alexander polynomials.

In Section 2, 3, and 4, we consider the twisted Alexander polynomial of each of these cases.
In Section 2, we consider the case (1), that is, two-bridge knots.
It is a special class of knots which contains twist knots and has many interesting properties \cite{BZ, H}.
They are alternating and algebraic knots, and have been completely classified \cite{S}. 
There are two kinds of famous projections of two-bridge knots and links, i.e. the Schubert presentations and the Conway presentations.
In this paper, we use the Conway presentations $C(2m_0, -2m_1, \ldots, 2m_{k-1}, -2m_k)$ where $k$ is odd; these presentation contains all two-bridge knots, and if $k$ is even, then they are two-bridge links. 
Since two-bridge knots and links are alternating, their genus are obtained from the degree of their Alexander polynomials \cite{C, Mu1, Mu2}.
For a two-bridge link (knot) $L=C(2c_1, 2c_2, \ldots, 2c_l)$, it is known that the genus $g(L)$ and the leading coefficient $\gamma(L)$ of the Alexander polynomial are given by
\begin{align*}
g(L) &=  \frac{1}{2} (l-\mu +1),\\
\gamma(L) &= \prod_{i=1}^l |c_i|,
\end{align*}
where $\mu$ is the number of the components of $L$ \cite{BZ}.
On the other hand, it is known that the twisted Alexander polynomials of two-bridge knots associated to their parabolic representations are nontrivial \cite{SW}.
The twisted Alexander polynomials of some class of two-bridge knots were computed, for example, genus one two-bridge knots which contains twist knots \cite{T1, MT}.
We compute the twisted Alexander polynomials of all two-bridge knots associated to their $SL_2(\bC)$-representations and obtain their leading coefficients and their degree explicitly.
We also get the Alexander polynomials and make sure that the result satisfies above equations.

In Section 3, we consider the case (2). 
The family of knots which satisfies the condition (2) is a huge family which contains all $(-2, 2n+1, 2m+1)$-pretzel knots.
The author computed the twisted Alexander polynomials of $(-2,3,2n+1)$-pretzel knot associated with their family of $SL_2(\bC)$-representations which contains their holonomy representations [A].
This case is quite important because the fiberedness and the genus of knots are not determined by their Alexander polynomials. 
It is known that there exist nonfibered knots with monic Alexander polynomials, and
since knots of the case (2) are not alternating, the degree of the Alexander polynomial is not always $2g$.
As in the case (1), we compute the twisted Alexander polynomials of all knots which satisfy the condition (2) associated to their $SL_2(\bC)$-representations and explicitly obtain their leading coefficients and their degree.
As a corollary, we show that the twisted Alexander polynomials of the exceptional cases i.e. nonfibered knots with monic Alexander polynomials, may have degree $4g-2$ and be non-monic polynomials.

In Section 4, we consider the case (3). 
Knots which satisfies the condition (3) are denoted by
\[
K_n=M(0;(3n+2,-2n-1), (3,1), (3,1))
\]
where $n$ is an integer \cite{MSY}.
In this case, the fiberedness of knots are determined by their Alexander polynomials.
We consider both cases $n$ is odd and even, and compute the twisted Alexander polynomials of each cases associated to their $SL_2(\bC)$-representations.
We explicitly obtain the degree and all the coefficients in this case.

{\it Acknowledgement} :
The author would like to thank Professor Yoshiyuki Yokota for supervising and giving helpful comments. She also would like to thank Professor Seiichi Kamada for giving valuable advice about the continued fractional expansions of two-bridge knots.

\section{The case (1)}
In this section, we give the presentation of knot groups of given two-bridge knots and compute their twisted Alexander polynomials associated to their $SL_2(\bC)$ representations.
Throughout this section, $K$ denotes the knot as in Figure 3, where $k$ is odd and  $2m_0, -2m_1, \cdots , -2m_k$ denotes the numbers of half twists in each of boxes.

\subsection{Main theorem}

\begin{figure}[h]
  \begin{center}
\includegraphics[clip,width=11cm]{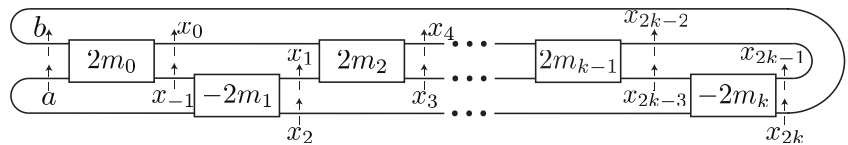}
\caption{Two-bridge knots}
  \end{center}
\end{figure}

If we put generators $a,b, x_{-1}, \cdots, x_{2k}$ as in Figure 3, the generators on the right hand side of each boxes are written by
\begin{align*}
x_{2i-1} &=
(x_{2i-3}^{(-1)^i} x_{2i-4}^{(-1)^i})^{m_i} x_{2i-3} (x_{2i-3}^{(-1)^i} x_{2i-4}^{(-1)^i})^{-m_i},\\
x_{2i} &=
 (x_{2i-3}^{(-1)^i} x_{2i-4}^{(-1)^i})^{m_i} x_{2i-4} (x_{2i-3}^{(-1)^i} x_{2i-4}^{(-1)^i})^{- m_i} , 
 \end{align*}
where we consider
$x_{-4}=b, \ x_{-3}=a, \ x_{-2}=a^{-1}$.
Then, we denote these relations by $r_{2i-1}$ and $r_{2i}$ for $0 \leq i \leq k$.
We also have two relations 
\begin{align*} 
x_{2k-1} =&\ x_{2k-2}^{-1},\\
x_{2k} =&\ b,
\end{align*}
and denote by $r_{2k+1}$ and $r_{2k+2}$.

Then, we have the following theorem.

\begin{thm}\label{2-bridge knot}
Let $K$ be a knot as in Figure 3, where $k$ is odd and $m_0, m_1, \cdots , m_k \in \bZ-\{ 0 \}$.
Then, the twisted Alexander polynomial $\Delta_{K,\rho}(t)$ of $K$ associated to a representation $\rho : G(K) \to SL_2(\bC)$ is given by
\[
\displaystyle \Delta_{K,\rho}(t) =  \lambda_0 t^0 + \cdots + \lambda_0  t^{2k},
\]
where
 \begin{eqnarray*}
\lambda_0= \prod^k_{i=0} \left| \sum^{|m_i|}_{j=1} \bigl(\rho(x_{2i-3}^{(-1)^i} x_{2i-4}^{(-1)^i})\bigr)^j \right|.
 \end{eqnarray*}
 \end{thm}

\subsection{Proof of Theorem \ref{2-bridge knot}}
Note that we have the knot group
\[
\pi_1(S^3-K) = \langle b, a, x_{-1}, \cdots, x_{2k}  \ | \   r_{-1} , \cdots r_{2k}, r_{2k+2} \rangle,
\]
and write $\Phi:=(\rho \otimes \frak{a}) \circ \phi$, where $\frak{a}$ and $\phi$ are as in Definition \ref{def of TAP}.
Then we put 
\begin{align*}
R_{-1}=&\ \Phi \left( \frac{\partial r_{-1}}{\partial a} \right) , \ 
R_{0}=  \Phi \left( \frac{\partial r_{0}}{\partial a} \right), \\
R_{1}=&\  \Phi \left( \frac{\partial r_{1}}{\partial a} \right), \ \ \  
R'_{1}=  \Phi \left( \frac{\partial r_{1}}{\partial x_{-1}} \right),\\
R_{2}=&\  \Phi \left( \frac{\partial r_{2}}{\partial a} \right), \ \ \ 
R'_{2}= \Phi \left( \frac{\partial r_{2}}{\partial x_{-1}} \right),
\end{align*}
and 
\begin{align*}
R_{2i-1}=&\ \Phi \left( \frac{\partial r_{2i-1}}{\partial x_{2i-4}} \right), \
R'_{2i-1}= \Phi \left( \frac{\partial r_{2i-1}}{\partial x_{2i-3}} \right), \\
R_{2i}=&\ \Phi \left( \frac{\partial r_{2i}}{\partial x_{2i-4}} \right), \ \ \ \ 
R'_{2i}= \Phi \left( \frac{\partial r_{2i}}{\partial x_{2i-3}} \right) ,
\end{align*}
for $2 \leq i \leq k$.

\begin{prop}\label{TAP of 2-bridge knot}
The twisted Alexander polynomial $\Delta_{K,\rho}(t)$ of $K$ associated to a representation $\rho : G(K) \to SL_2(\bC)$ is given by
\[
\displaystyle \Delta_{K,\rho}(t) = \frac{|N_{2k}|}{t^2 - \mathrm{tr} \rho(b) t +1},
\]
where the matrix $N_{2k}$ is defined by
 \begin{align*}
N_0 &= R_0,\\
N_1 &= R_1 + (-R'_1) R_{-1},\\
N_2 &= R_2 + (-R'_2) R_{-1},
 \end{align*}
 and
  \begin{align*}
N_{2i-1} &= (-R_{2i-1} )N_{2i-4} + (-R'_{2i-1}) N_{2i-3},\\
N_{2i} &= (-R_{2i} )N_{2i-4} + (-R'_{2i}) N_{2i-3},
 \end{align*}
for $2 \le i \le k$.
\end{prop}

\begin{proof}

Let $A_{\rho,b}$ be the matrix obtained by removing the first 2 columns of twisted Alexander matrix $A_{\rho}$, whose size is $2(2k + 3) \times 2(2k + 4)$.

Then, by the definition of twisted Alexander polynomials, 
\[
\Delta_{K,\rho}(t) = \frac{|A_{\rho,b}|}{|\Phi(b-1)|}= \frac{|A_{\rho,b}|}{t^2 - \mathrm{tr} \rho(b) t +1}.
\]

Now, we can transform $|A_{\rho,b}|$ by
\begin{align*}
|A_{\rho,b}|
&= \left|
\begin{array}{ccccccccccc}
R_{-1} & E & O & O &  \ldots & O & O & O & O & O\\ 
R_{0} & O & E & O & \ldots  & O & O & O & O & O\\ 
R_{1} & R'_{1} & O & E &  & O & O & O & O & O\\ 
R_{2} & R'_{2} & O & O &  \ddots & O & O & O & O & O\\ 
O & O & R_{3} & R'_{3} &   & O & O & O & O & O\\ 
O & O & R_{4} & R'_{4} &  & O & O & O & O & O\\
 \vdots & \vdots &  &  & \ddots &  &  & \ddots &  & \vdots\\
O & O & O & O &  & R_{2k-1} & R'_{2k-1} & O & E & O\\ 
O & O & O & O &  & R_{2k} & R'_{2k} & O & O & E\\
O & O & O & O &  \ldots & O & O & O & O & E\\
\end{array}
\right| \\
&=
\left|
\begin{array}{ccccccccccccc}
R_{-1} & E & O & O &   \ldots & O & O & O & O & O\\ 
N_{0} & O & E & O &  \ldots  & O & O & O & O & O\\ 
N_{1} & O & O & E &  & O & O & O & O & O\\ 
N_{2} & O & O & O &  \ddots & O & O & O & O & O\\ 
N_{3} & O & O & O &    & O & O & O & O & O\\ 
N_{4} & O & O & O &  & O & O & O & O & O\\
 \vdots & \vdots & \vdots & \vdots  &  & \vdots & \vdots & \ddots &  & \vdots\\
N_{2k-1} & O & O & O &  \ldots & O & O & O & E & O\\ 
N_{2k} & O & O & O &  \ldots & O & O & O & O & E\\
O & O & O & O &  \ldots & O & O & O & O & E\\
\end{array}
\right| \\
&=  |N_{2k}|,
\end{align*}
where
 \begin{align*}
N_0 &= R_0,\\
N_1 &= R_1 + (-R'_1) R_{-1},\\
N_2 &= R_2 + (-R'_2) R_{-1},
 \end{align*}
 and
\begin{align*}
N_{2i-1} &= (-R_{2i-1} )N_{2i-4} + (-R'_{2i-1}) N_{2i-3},\\
N_{2i} &= (-R_{2i} )N_{2i-4} + (-R'_{2i}) N_{2i-3},
\end{align*}
for $2 \le i \le k$.
\end{proof}

Now, we compute the matrix $N_{2k}$.
The following proposition gives the highest and the lowest degree of each $N_{2i-1}$ and $N_{2i}$.
\begin{prop}\label{N_i}
If $i$ is even, we have
\[
N_{2i-1} = \sum_{j=-\frac{i}{2}-1}^{\frac{i}{2}} t^j N_{2i-1}^j , \ 
N_{2i} = \sum_{j=-\frac{i}{2}}^{\frac{i}{2}+1} t^j N_{2i}^j ,
\]
and if $i$ is odd, we have
\[
N_{2i-1} = \sum_{j=-\frac{i+1}{2}}^{\frac{i+1}{2}} t^j N_{2i-1}^j , \ 
N_{2i} = \sum_{j=-\frac{i+1}{2}+1}^{\frac{i+1}{2}+1} t^j N_{2i}^j .
\]
\end{prop}
\begin{proof}

If we put $A=\rho(a)$, $B=\rho(b)$ and $X_i=\rho(x_i)$ for $i=-1, \cdots , 2k$, 
then we have
\begin{align*}
N_0 &= R_0 = t^0 R_0^0 + t^1 R_0^1
=: t^0 N_0^0 + t^1 N_0^1.
 \end{align*}
If we put
\begin{align*}
R_{-1}&=t^{-1}R_{-1}^{-1}+t^0R_{-1}^{0},\\
R_{1}&=t^{0}R_{1}^{0}+t^1R_{1}^{1},\\
-R'_{1}&=t^{0}{R'}_{1}^{0}+t^1{R'}_{1}^{1},\\
R_{2}&=t^{1}R_{2}^{1}+t^2R_{2}^{2},\\
-R'_{2}&=t^{1}{R'}_{2}^{1}+t^2{R'}_{2}^{2},
\end{align*}
we have
\begin{align*}
 N_1 &= R_1 + (-R'_1) R_{-1}\\ 
&=t^{-1} {R'}_1^0 R_{-1}^{-1} + t^0 (R_1^0 + {R'}_1^0 R_{-1}^0 + {R'}_1^1 R_{-1}^{-1}) + t^1 (R_1^1 + {R'}_1^1 R_{-1}^0)
 =: t^{-1} N_1^{-1} + t^0 N_1^0 +t^1 N_1^1,\\
N_2 &= R_2 + (-R'_2) R_{-1} \\ 
&=t^{0} {R'}_2^1 R_{-1}^{-1} + t^1 (R_2^1 + {R'}_2^1 R_{-1}^0 + {R'}_2^2 R_{-1}^{-1}) + t^2 (R_2^2 + {R'}_2^2 R_{-1}^0)
=: t^0 N_2^0 + t^1 N_2^1 +t^2 N_2^2.
 \end{align*}
This proves the statement for $i=1$.

Suppose that $N_{0}, N_{1} , \ldots, N_{2h-3}$ satisfy the statement .
If $h$ is even, since we have
\begin{align*}
-R_{2i-1}&=t^{-2}R_{2i-1}^{-2}+t^{-1}R_{2i-1}^{-1},\\
-R'_{2i-1}&=t^{-1}{R'}_{2i-1}^{-1}+t^0{R'}_{2i-1}^{0},\\
-R_{2i}&=t^{-1}R_{2i}^{-1}+t^0 R_{2i}^{0},\\
-R'_{2i}&=t^{0}{R'}_{2i}^{0}+t^1 {R'}_{2i}^{1},
\end{align*}
for $i$ is even and $2 \le i \le k$, then we get
\begin{align*}
N_{2h-1} &= (-R_{2h-1} )N_{2h-4} + (-R'_{2h-1}) N_{2h-3}\\
&= (t^{-2}R_{2h-1}^{-2}+t^{-1}R_{2h-1}^{-1})N_{2(h-2)} + (t^{-1}{R'}_{2h-1}^{-1}+t^0{R'}_{2h-1}^{0}) N_{2(h-1)-1}\\
&= (t^{-2}R_{2h-1}^{-2}+t^{-1}R_{2h-1}^{-1})\sum_{j=-\frac{h-2}{2}}^{\frac{h-2}{2}+1} t^j N_{2(h-2)}^j + (t^{-1}{R'}_{2h-1}^{-1}+t^0{R'}_{2h-1}^{0}) \sum_{j=-\frac{h-1}{2}-1}^{\frac{h-1}{2}} t^j N_{2(h-1)-1}^j\\
&= t^{-\frac{h}{2}-1} (R_{2h-1}^{-2} N_{2(h-2)}^{-\frac{h}{2}+1} + {R'}_{2h-1}^{-1} N_{2(h-1)-1}^{-\frac{h}{2}} ) + \cdots + t^{\frac{h}{2}} {R'}_{2h-1}^{0} N_{2(h-1)-1}^{\frac{h}{2}}\\
&=: t^{-\frac{h}{2}-1} N_{2h-1}^{-\frac{h}{2}-1} + \cdots + t^{\frac{h}{2}} N_{2h-1}^{\frac{h}{2}}.
\end{align*}
Similarly, we have
\begin{align*}
N_{2h} &= (-R_{2h} )N_{2h-4} + (-R'_{2h}) N_{2h-3}\\
&= t^{-\frac{h}{2}} (R_{2h}^{-1} N_{2(h-2)}^{-\frac{h}{2}+1} + {R'}_{2h}^{0} N_{2(h-1)-1}^{-\frac{h}{2}} ) + \cdots + t^{\frac{h}{2}+1} {R'}_{2h}^{1} N_{2(h-1)-1}^{\frac{h}{2}}\\
&=: t^{-\frac{h}{2}} N_{2h}^{-\frac{h}{2}-1} + \cdots + t^{\frac{h}{2}+1} N_{2h}^{\frac{h}{2}}.
\end{align*}
This shows the statement in the case when $i$ is even.

If $h$ is odd, since we have
\begin{align*}
-R_{2i-1}&=t^{-1}R_{2i-1}^{-1}+t^{0}R_{2i-1}^{0},\\
-R'_{2i-1}&=t^{0}{R'}_{2i-1}^{0}+t^1 {R'}_{2i-1}^{1},\\
-R_{2i}&=t^{0}R_{2i}^{0}+t^1 R_{2i}^{1},\\
-R'_{2i}&=t^{1}{R'}_{2i}^{1}+t^2 {R'}_{2i}^{2},
\end{align*}
for $i$ is odd and $2 \le i \le k$, then we get
\begin{align*}
N_{2h-1} &= (-R_{2h-1} )N_{2h-4} + (-R'_{2h-1}) N_{2h-3}\\
&= (t^{-1}R_{2h-1}^{-1}+t^{0}R_{2h-1}^{0})N_{2(h-2)} + (t^{0}{R'}_{2h-1}^{0}+t^1{R'}_{2h-1}^{1}) N_{2(h-1)-1}\\
&= (t^{-1}R_{2h-1}^{-1}+t^{0}R_{2h-1}^{0}) \sum_{j=-\frac{(h-2)+1}{2}+1}^{\frac{(h-2)+1}{2}+1} t^j N_{2(h-2)}^j + (t^{0}{R'}_{2h-1}^{0}+t^1{R'}_{2h-1}^{1}) \sum_{j=-\frac{h-1}{2}-1}^{\frac{h-1}{2}} t^j N_{2(h-1)-1}^j\\
&= t^{-\frac{h+1}{2}} {R'}_{2h-1}^{0} N_{2(h-1)-1}^{-\frac{h+1}{2}}  + \cdots + t^{\frac{h+1}{2}} (R_{2h-1}^{0} N_{2(h-2)}^{\frac{h+1}{2}} + {R'}_{2h-1}^{1} N_{2(h-1)-1}^{\frac{h-1}{2}})\\
&=: t^{-\frac{h+1}{2}} N_{2h-1}^{-\frac{h+1}{2}} + \cdots + t^{\frac{h+1}{2}}  N_{2h-1}^{\frac{h+1}{2}}.
\end{align*}
Similarly, we have
\begin{align*}
N_{2h} &= (-R_{2h} )N_{2h-4} + (-R'_{2h}) N_{2h-3}\\
&= t^{-\frac{h+1}{2}+1} {R'}_{2h}^{1} N_{2(h-1)-1}^{-\frac{h+1}{2}}  + \cdots + t^{\frac{h+1}{2}+1} (R_{2h}^{1} N_{2(h-2)}^{\frac{h+1}{2}} + {R'}_{2h}^{2} N_{2(h-1)-1}^{\frac{h-1}{2}})\\
&=: t^{-\frac{h+1}{2}+1} N_{2h}^{-\frac{h+1}{2}+1} + \cdots + t^{\frac{h+1}{2}+1}  N_{2h}^{\frac{h+1}{2}+1}.
\end{align*}
This shows the statement in the case when $i$ is odd.
\end{proof}

\hspace{2mm}
Then, we put
 \begin{align*}
N_{2i-1}^{\max}
&= 
\begin{cases}
N_{2i-1}^{\frac{i}{2}} = {R'}_{2i-1}^{0} N_{2(i-1)-1}^{\frac{i}{2}}& $if$ \  i \ $is even$\\
N_{2i-1}^{\frac{i+1}{2}} = R_{2i-1}^{0} N_{2(i-2)}^{\frac{i+1}{2}} + {R'}_{2i-1}^{1} N_{2(i-1)-1}^{\frac{i-1}{2}} & $if$ \  i \ $is odd$\\
\end{cases},\\
N_{2i}^{\max}
&= 
\begin{cases}
N_{2i}^{\frac{i}{2}+1} = {R'}_{2i}^{1} N_{2(i-1)-1}^{\frac{i}{2}} & $if$ \  i \ $is even$\\
N_{2i}^{\frac{i+1}{2}+1} = R_{2i}^{1} N_{2(i-2)}^{\frac{i+1}{2}} + {R'}_{2i}^{2} N_{2(i-1)-1}^{\frac{i-1}{2}} & $if$ \ i \ $is odd$\\
\end{cases},\\
N_{2i-1}^{\min}
&= 
\begin{cases}
N_{2i-1}^{-\frac{i}{2}-1} = R_{2i-1}^{-2} N_{2(i-2)}^{-\frac{i}{2}+1} + {R'}_{2i-1}^{-1} N_{2(i-1)-1}^{-\frac{i}{2}} & $if$ \ i \ $is even$\\
N_{2i-1}^{-\frac{i+1}{2}} = {R'}_{2i-1}^{0} N_{2(i-1)-1}^{-\frac{i+1}{2}} & $if$ \ i \ $is odd$\\
\end{cases},\\
N_{2i}^{\min} 
&= 
\begin{cases}
N_{2i}^{-\frac{i}{2}} = R_{2i}^{-1} N_{2(i-2)}^{-\frac{i}{2}+1} + {R'}_{2i}^{0} N_{2(i-1)-1}^{-\frac{i}{2}} & $if$ \ i \ $is even$\\
N_{2i}^{-\frac{i+1}{2}+1} = {R'}_{2i}^{1} N_{2(i-1)-1}^{-\frac{i+1}{2}} & $if$ \ i \ $is odd$\\
\end{cases}.
 \end{align*}
Since we have
\begin{align*}
|N_{2k}| =& \left|  \sum_{j=-\frac{k+1}{2}+1}^{\frac{k+1}{2}+1} t^j N_{2k}^j  \right|\\
=& t^{2(-\frac{k+1}{2}+1)} \Bigl|N_{2k}^{-\frac{k+1}{2}+1} \Bigr| + \cdots + t^{2(\frac{k+1}{2}+1)} \Bigl|N_{2k}^{\frac{k+1}{2}+1} \Bigr|\\
=& t^{-k+1} |N_{2k}^{\min}| + \cdots + t^{k+3} |N_{2k}^{\max}| ,
\end{align*}
we obtain 
\begin{align*}
\displaystyle \Delta_{K,\rho}(t) =& \frac{|N_{2k}|}{t^2 - \mathrm{tr} \rho(b) t +1}\\
=& \frac{t^{-k+1} |N_{2k}^{\min} | + \cdots + t^{k+3} |N_{2k}^{\max}|}{t^2 - \mathrm{tr} \rho(b) t +1}\\
\sim & t^{0} |N_{2k}^{\min} | + \cdots + t^{2k} |N_{2k}^{\max}|.
\end{align*}

Now we compute $|N_{2k}^{\max}|$ and $|N_{2k}^{\min}|$.
If $i$ is odd, since
\begin{align*}
N_{2i-1}^{\max} &= R_{2i-1}^{0} N_{2(i-2)}^{\max} + {R'}_{2i-1}^{1} N_{2(i-1)-1}^{\max}\\
 &= - \frac{|m_i|}{m_i} \sum_{j}^{} (X_{2i-3}^{-1} X_{2i-4}^{-1})^j  N_{2(i-2)}^{\max} - \frac{|m_i|}{m_i} \sum_{j}^{} (X_{2i-3}^{-1} X_{2i-4}^{-1})^j X_{2i-4}N_{2(i-1)-1}^{\max}\\
&= - \frac{|m_i|}{m_i} \sum_{j}^{} (X_{2i-3}^{-1} X_{2i-4}^{-1})^j  \Bigl\{N_{2(i-2)}^{\max} + X_{2i-4} N_{2(i-1)-1}^{\max} \Bigr\},\\
  N_{2i}^{\max} &= R_{2i}^{1} N_{2(i-2)}^{\max} + {R'}_{2i}^{2} N_{2(i-1)-1}^{\max}\\
&=  \frac{|m_i|}{m_i} X_{2i}\sum_{j}^{} (X_{2i-3}^{-1} X_{2i-4}^{-1})^j N_{2(i-2)}^{\max} + \frac{|m_i|}{m_i}X_{2i} \sum_{j}^{} (X_{2i-3}^{-1} X_{2i-4}^{-1})^j X_{2i-4} N_{2(i-1)-1}^{\max}\\
 &=  \frac{|m_i|}{m_i} X_{2i} \sum_{j}^{} (X_{2i-3}^{-1} X_{2i-4}^{-1})^j \Bigl\{N_{2(i-2)}^{\max} + X_{2i-4} N_{2(i-1)-1}^{\max} \Bigr\},
\end{align*}
we have
\[
N_{2i}^{\max} = -X_{2i} N_{2i-1}^{\max}.
\]
By using this relation, we get
\begin{align*}
 N_{2i-1}^{\max} =& - \frac{|m_i|}{m_i} \sum_{j}^{} (X_{2i-3}^{-1} X_{2i-4}^{-1})^j  \Bigl\{N_{2(i-2)}^{\max} + X_{2i-4} N_{2(i-1)-1}^{\max} \Bigr\}\\
 =& - \frac{|m_i|}{m_i} \sum_{j}^{} (X_{2i-3}^{-1} X_{2i-4}^{-1})^j  \Bigl\{ (-X_{2(i-2)} N_{2(i-2)-1}^{\max}) + X_{2i-4} ( {R'}_{2(i-1)-1}^{0} N_{2((i-1)-1)-1}^{\max}) \Bigr\}\\
=& - \frac{|m_i|}{m_i} \sum_{j}^{} (X_{2i-3}^{-1} X_{2i-4}^{-1})^j X_{2i-4} \Bigl\{ -E+ {R'}_{2(i-1)-1}^{0}  \Bigr\} N_{2(i-2)-1}^{\max}\\ 
=& - \frac{|m_i|}{m_i} \sum_{j}^{} (X_{2i-3}^{-1} X_{2i-4}^{-1})^j X_{2i-4} \\
&\Bigl\{ -E+ \frac{|m_{i-1}|}{m_{i-1}} X_{2i-5}^{-1}\Bigl(\sum_{j}^{} (X_{2i-5} X_{2i-6})^j + \frac{|m_{i-1}|}{m_{i-1}}  (X_{2i-5} X_{2i-6})^{m_{i-1}+1}\Bigr) X_{2i-6}^{-1}  \Bigr\} N_{2(i-2)-1}^{\max}\\ 
=& - \frac{|m_i|}{m_i} \sum_{j}^{} (X_{2i-3}^{-1} X_{2i-4}^{-1})^j X_{2i-4} \\
&X_{2i-5}^{-1} \Bigl\{ -X_{2i-5}X_{2i-6}+ \frac{|m_{i-1}|}{m_{i-1}} \sum_{j}^{} (X_{2i-5} X_{2i-6})^j + (X_{2i-5} X_{2i-6})^{m_{i-1}+1} \Bigr\} X_{2i-6}^{-1} N_{2(i-2)-1}^{\max}\\ 
=& - \frac{|m_i|}{m_i} \sum_{j}^{} (X_{2i-3}^{-1} X_{2i-4}^{-1})^j X_{2i-4} \\
&X_{2i-5}^{-1} \Bigl\{ -X_{2i-5}X_{2i-6}+ \frac{|m_{i-1}|}{m_{i-1}} \sum_{j}^{} (X_{2i-5} X_{2i-6})^j + (X_{2i-5} X_{2i-6})^{m_{i-1}+1} \Bigr\} X_{2i-6}^{-1} N_{2(i-2)-1}^{\max}\\ 
=& 
\begin{cases}
\displaystyle - \frac{|m_i|}{m_i} \sum_{j}^{} (X_{2i-3}^{-1} X_{2i-4}^{-1})^j X_{2i-4} X_{2i-5}^{-1}\Bigl\{ \sum_{j=2}^{m_{i-1}+1} (X_{2i-5} X_{2i-6})^j \Bigr\} X_{2i-6}^{-1} N_{2(i-2)-1}^{\max} & $if$ \  m_{i-1}>0\\\\ 
\displaystyle - \frac{|m_i|}{m_i} \sum_{j}^{} (X_{2i-3}^{-1} X_{2i-4}^{-1})^j X_{2i-4}
X_{2i-5}^{-1}\Bigl\{ - \sum_{j = m_{i-1}+2}^{1} (X_{2i-5} X_{2i-6})^j \Bigr\} X_{2i-6}^{-1} N_{2(i-2)-1}^{\max} & $if$ \  m_{i-1}<0\\
\end{cases}\\
&= - \frac{|m_i m_{i-1}|}{m_i m_{i-1}} \sum_{j}^{} (X_{2i-3}^{-1} X_{2i-4}^{-1})^j X_{2i-4} \sum_{j}^{} (X_{2i-5} X_{2i-6})^j  N_{2(i-2)-1}^{\max}
 \end{align*}
 for $i$ is odd.

Since $k$ is odd, 
\begin{align*}
N_{2k}^{\max} =& -X_{2k} N_{2k-1}^{\max}\\
=& - X_{2k}\\
&
\Bigl\{- \frac{|m_k m_{k-1}|}{m_k m_{k-1}} \sum_{j}^{} (X_{2k-3}^{-1} X_{2k-4}^{-1})^j X_{2k-4} \sum_{j}^{} (X_{2k-5} X_{2k-6})^j \Bigr\}\\
&\ \ \ \ \ \ \ \ \ \ \ \ \ \ \  \vdots \\
&\Bigl\{- \frac{|m_3 m_{2}|}{m_3 m_{2}} \sum_{j}^{} (X_{3}^{-1} X_{2}^{-1})^j X_{2} \sum_{j}^{} (X_{1} X_{0})^j \Bigr\}\\ &N_1^{\max}\\
=& - X_{2k}\\
&\Bigl\{- \frac{|m_k m_{k-1}|}{m_k m_{k-1}} \sum_{j}^{} (X_{2k-3}^{-1} X_{2k-4}^{-1})^j X_{2k-4} \sum_{j}^{} (X_{2k-5} X_{2k-6})^j \Bigr\}\\
&\ \ \ \ \ \ \ \ \ \ \ \ \ \ \  \vdots \\
&\Bigl\{- \frac{|m_3 m_{2}|}{m_3 m_{2}} \sum_{j}^{} (X_{3}^{-1} X_{2}^{-1})^j X_{2} \sum_{j}^{} (X_{1} X_{0})^j \Bigr\} \\
&\Bigl\{- \frac{|m_1|}{m_1} \sum_{j}^{} (X_{-1}^{-1} A)^j A^{-1} + \Bigl(- \frac{|m_1|}{m_1} \sum_{j}^{} (X_{-1}^{-1} A)^j A^{-1} \Bigr) \Bigl(- \frac{|m_0|}{m_0} \bigl(\sum_{j}^{} (A B)^j (A B)^{-1}+ (AB)^{m_0} \bigr) \Bigr) \Bigr\}
\end{align*}
\begin{align*}
=& (-1)^{\frac{k+1}{2}} X_{2k} \frac{|m_k m_{k-1}\cdots m_0|}{m_k m_{k-1} \cdots m_0}\\
&\Bigl\{  \sum_{j}^{} (X_{2k-3}^{-1} X_{2k-4}^{-1})^j X_{2k-4} \sum_{j}^{} (X_{2k-5} X_{2k-6})^j \Bigr\} \cdots 
\Bigl\{ \sum_{j}^{} (X_{3}^{-1} X_{2}^{-1})^j X_{2} \sum_{j}^{} (X_{1} X_{0})^j \Bigr\} \\
&\Bigl\{  \sum_{j}^{} (X_{-1}^{-1} A)^j A^{-1} \sum_{j}^{} (A B)^j \Bigr\},
\end{align*}
where $j$ are given by
\begin{align*}
j=
\begin{cases}
 1, \cdots , m_i & $if$ \  m_{i-1}>0\\\\ 
 m_0+1 , \cdots, 0 & $if$ \  m_{i-1}<0\\
\end{cases},
 \end{align*}
 for each $i$.
Note that $X_{-4}=B, \ X_{-3}=A, \ X_{-2}=A^{-1}$ and $X_i \in SL_2(\bC)$ for all $i$. 
Then we obtain
\begin{align*}
|N_{2k}^{\max}|=
&\left|  \sum_{j}^{} (X_{2k-3}^{-1} X_{2k-4}^{-1})^j \right| \left| \sum_{j}^{} (X_{2k-5} X_{2k-6})^j \right| \cdots 
\left| \sum_{j}^{} (X_{3}^{-1} X_{2}^{-1})^j \right| \left| \sum_{j}^{} (X_{1} X_{0})^j \right| \\
&\left| \sum_{j}^{} (X_{-1}^{-1} A)^j \right| \left| \sum_{j}^{} (A B)^j \right|\\
=
&\left|  \sum_{j=1}^{|m_{k}|} (X_{2k-3}^{-1} X_{2k-4}^{-1})^j \right| \left| \sum_{j=1}^{|m_{k-1}|} (X_{2k-5} X_{2k-6})^j \right| \cdots 
\left| \sum_{j=1}^{|m_{3}|} (X_{3}^{-1} X_{2}^{-1})^j \right| \left| \sum_{j=1}^{|m_{2}|} (X_{1} X_{0})^j \right| \\
&\left| \sum_{j=1}^{|m_{1}|} (X_{-1}^{-1} A)^j \right| \left| \sum_{j=1}^{|m_{0}|} (A B)^j \right|\\
=&\prod^k_{i=0} \left| \sum^{|m_i|}_{j=1} (X_{2i-3}^{(-1)^i} X_{2i-4}^{(-1)^i})^j \right|.
\end{align*}

Similary, we have 
 \begin{align*}
 N_{2i-1}^{\min} &=   
- \frac{|m_i|}{m_i} X_{2i-1} \Bigl( \sum_{j} (X_{2i-3} X_{2i-4})^j \Bigr) X_{2i-4}^{-1} (N_{2i-4}^{\min}+X_{2i-3}^{-1}N_{2i-3}^{\min}), \\
 N_{2i}^{\min} &= -X_{2i-1}^{-1} N_{2i-1}^{\min} ,
\end{align*}
if $i$ is even, and
\begin{align*}
 N_{2i-1}^{\min} &=  
\begin{cases}
\displaystyle \frac{|m_i|}{m_i} X_{2i-1} \Bigl( \sum_{j} (X_{2i-3}^{-1} X_{2i-4}^{-1})^j \Bigr) X_{2i-4} N_{2i-3}^m & $if$ \ m_i \neq -1\\
O & $if$ \ m_i =-1\\
\end{cases}, \\
 N_{2i}^{\min} &= -\frac{|m_i|}{m_i} \Bigl( \sum_{j} (X_{2i-3}^{-1} X_{2i-4}^{-1})^j \Bigr) X_{2i-4} N_{2i-3}^{\min} ,
\end{align*}
if $i$ is odd.
By the above relations, if $i$ is even, we have
 \begin{align*}
N_{2i-1}^{\min} = -\frac{|m_i m_{i-1}|}{m_i m_{i-1}} X_{2i-1} \Bigl( \sum_{j} (X_{2i-3} X_{2i-4})^j \Bigr)  X_{2i-4}^{-1} \Bigl( \sum_{j} (X_{2i-5}^{-1} X_{2i-6}^{-1})^j \Bigr) X_{2i-5}^{-1} N_{2i-5}^{\min}.
 \end{align*}
Hence, we obtain
 \begin{align*}
N_{2k}^{\min} =&\ (-1)^{\frac{k+1}{2}}\frac{|m_k \cdots m_0|}{m_k \cdots m_0} 
\Bigl(\sum_{j} (X_{2k-3}^{-1} X_{2k-4}^{-1})^j \Bigr) \\
& \biggl\{ \Bigl( \sum_{j} (X_{2k-5} X_{2k-6})^j \Bigr) X_{2k-6}^{-1} \Bigl( \sum_{j} (X_{2k-7}^{-1} X_{2k-8}^{-1})^j \Bigr) \biggr\}  \\
&\ \ \ \ \ \ \ \ \ \ \ \ \ \ \  \vdots \\
& \biggl\{ \Bigl( \sum_{j} (X_{1} X_{0})^j \Bigr) X_{0}^{-1} \Bigl( \sum_{j} (X_{-1}^{-1} A)^j \Bigr)\biggr\}\\
&
\Bigl( \sum_{j} (A B)^j \Bigr),
 \end{align*}
and 
 \begin{align*}
| N_{2k}^{\min} |
=\prod^k_{i=0} \left| \sum^{|m_i|}_{j=1} (X_{2i-3}^{(-1)^i} X_{2i-4}^{(-1)^i})^j \right|.
\end{align*}
This completes the proof of Theorem \ref{2-bridge knot}.
 \qed

 \vspace{5mm}
 
 We can get the leading coefficient of the Alexander polynomial of $K$ from above discussion.

\begin{rem}
If we denote the Alexander polynomial of $K$ by $\Delta_K(t)$, then we have
\begin{align*}
\Delta_{K} (t)=& |N_{2k}|
=\left|  \sum_{j=-\frac{k+1}{2}+1}^{\frac{k+1}{2}+1} t^j N_{2k}^j  \right|\\
=& t^{-\frac{k+1}{2}+1} N_{2k}^{-\frac{k+1}{2}+1}  + \cdots + t^{\frac{k+1}{2}+1} N_{2k}^{\frac{k+1}{2}+1} .
\end{align*}
In this case, each $N_{2k}^j$ are $1 \times 1$ matrices and all $X_i$ are the identity matrix $1$.
Same as in the case of the twisted Alexander polynomial, we can compute $N_{2k}^{\max}$ and $N_{2k}^{\min}$, that is, we get
\begin{align*}
N_{2k}^{\max}=&
 (-1)^{\frac{k+1}{2}} X_{2k} \frac{|m_k m_{k-1}\cdots m_0|}{m_k m_{k-1} \cdots m_0}\\
&\Bigl\{  \sum_{j}^{} (X_{2k-3}^{-1} X_{2k-4}^{-1})^j X_{2k-4} \sum_{j}^{} (X_{2k-5} X_{2k-6})^j \Bigr\} \cdots 
\Bigl\{ \sum_{j}^{} (X_{3}^{-1} X_{2}^{-1})^j X_{2} \sum_{j}^{} (X_{1} X_{0})^j \Bigr\} \\
&\Bigl\{  \sum_{j}^{} (X_{-1}^{-1} A)^j A^{-1} \sum_{j}^{} (A B)^j \Bigr\}\\
=&  (-1)^{\frac{k+1}{2}} \frac{|m_k m_{k-1}\cdots m_0|}{m_k m_{k-1} \cdots m_0} |m_k| |m_{k-1}|\cdots |m_0|\\
=&  (-1)^{\frac{k+1}{2}} m_k m_{k-1}\cdots m_0
\end{align*}
and
\begin{align*}
N_{2k}^{\min}=&
 (-1)^{\frac{k+1}{2}}\frac{|m_k \cdots m_0|}{m_k \cdots m_0} 
\Bigl(\sum_{j} (X_{2k-3}^{-1} X_{2k-4}^{-1})^j \Bigr) \\
& \biggl\{ \Bigl( \sum_{j} (X_{2k-5} X_{2k-6})^j \Bigr) X_{2k-6}^{-1} \Bigl( \sum_{j} (X_{2k-7}^{-1} X_{2k-8}^{-1})^j \Bigr) \biggr\}  \\
&\ \ \ \ \ \ \ \ \ \ \ \ \ \ \  \vdots \\
& \biggl\{ \Bigl( \sum_{j} (X_{1} X_{0})^j \Bigr) X_{0}^{-1} \Bigl( \sum_{j} (X_{-1}^{-1} A)^j \Bigr)\biggr\}\\
&
\Bigl( \sum_{j} (A B)^j \Bigr)\\
=&  (-1)^{\frac{k+1}{2}} \frac{|m_k m_{k-1}\cdots m_0|}{m_k m_{k-1} \cdots m_0} |m_k| |m_{k-1}|\cdots |m_0|\\
=&  (-1)^{\frac{k+1}{2}} m_k m_{k-1}\cdots m_0.
 \end{align*}
Hence we can write
\begin{align*}
\Delta_{K} (t)= \kappa_0 t^0+ \cdots + \kappa_0 t^{k+1},
\end{align*}
where
\begin{eqnarray*}
\kappa_0 =
|m_k \cdots m_0|.
\end{eqnarray*}
Thus, since the genus of $K$ is given by $\frac{k+1}{2}$, we have
\[
\deg(\Delta_{K}(t)) = 2 g(K).
\]
These results coincide with the results of \cite{BZ}.
Furthermore, we have
\[
\deg(\Delta_{K,\rho}(t)) = 4 g(K) -2,
\]
if the representation $\rho: G(K) \to SL_2(\bC)$ satisfies
\[
\left| \sum^{|m_i|}_{j=1} (X_{2i-3}^{(-1)^i} X_{2i-4}^{(-1)^i})^j \right| \neq 0,
\]
for $i=0,1, \ldots, k$.
This condition means that any eigenvalue $\lambda$ of $2 \times 2$ matrices $X_{2i-3}^{(-1)^i} X_{2i-4}^{(-1)^i}$ satisfies
\[
\lambda^{m_i} \neq  1.
\]
\end{rem}

\section{The case (2)}
Throughout this section, we suppose $K=M(b;(\alpha_1,\beta_1), (\alpha_2, \beta_2),  (\alpha_3,\beta_3))$ which satisfies condition (2).
In this section, we give the presentation of knot groups of $K$ and compute their twisted Alexander polynomials associated to their $SL_2(\bC)$ representations.

To this end, we have the following results.

\begin{lem}\label{the continued fraction expansions}
For integers $\alpha$ and $\beta$, we have the continued fraction expansions
\begin{align*}
\frac{\beta}{\alpha}=&
\begin{cases}
\displaystyle 2 m_0 + \frac{1}{\displaystyle 2 m_1 + \frac{}{ \ddots +\displaystyle \frac{1}{2 m_{2l}}}} & $if$ \ \alpha \ $is odd and$ \ \beta \ $is even$,\\
\displaystyle 2 m_0 + \frac{1}{\displaystyle 2 m_1 + \frac{}{ \ddots +\displaystyle \frac{1}{2 m_{2l+1}}}} & $if$ \ \alpha \ $is even and$ \ \beta \ $is odd$,\\
\displaystyle 2 m_0 + \frac{1}{\displaystyle 2 m_1 + \frac{}{ \ddots +\displaystyle \frac{1}{2 m_{l}+1}}} & $if$ \ \alpha \ $and$ \ \beta \ $is odd$,
\end{cases}
\end{align*}
for some non-zero $m_i \in \bZ$.
\end{lem}

\begin{prop}\label{the continued fraction expansions of case 2}
$K$ can be transformed to a Montesinos knot $M(0;(\alpha_1,\beta_1), (\alpha_2, \beta_2),  (\alpha_3,\beta_3))$ whose three rational tangles $\beta_1/\alpha_1$, $\beta_2/\alpha_2$ and $\beta_3/\alpha_3$ are written by
\begin{align*}
\frac{\beta_1}{\alpha_1}=& \pm \frac{1}{2},\\
\frac{\beta_2}{\alpha_2}=& 2 m_0 + \frac{1}{2 m_1+\displaystyle\frac{1}{2 m_2 +  \displaystyle\frac{1}{ \ddots +\displaystyle \frac{1}{2 m_k}}}},\\
\frac{\beta_3}{\alpha_3}=& \frac{1}{2 n_1+\displaystyle\frac{1}{2 n_2 +  \displaystyle\frac{1}{ \ddots +\displaystyle\frac{1}{2 n_l}}}},
\end{align*}
where both $k$ and $l$ are even and $\alpha_1, \alpha_2, \alpha_3$ are positive integers.  
\end{prop}

\begin{proof}
For a Montesinos knot $K=M(b;(\alpha_1,\beta_1), (\alpha_2, \beta_2),  (\alpha_3,\beta_3))$, 
we can put $b$ half twists into a tangle $\beta_1/\alpha_1$, i.e. we may suppose $b=0$.
Then, since $\alpha_1=2$, we have
\begin{align*}
\frac{\beta_1}{\alpha_1}=& \frac{2m+1}{2}
=m +\frac{1}{2},
\end{align*}
and we can move $m$ half twists to tangle $\beta_2/\alpha_2$.
Hence we put
\[
\frac{\beta_1'}{\alpha_1}:= \frac{\beta_1}{\alpha_1} -m = \frac{1}{2}.
\]
Similary, by putting some twists from the tangle $\beta_3/\alpha_3$ into the tangle $\beta_2/\alpha_2$ if necessary, 
 we suppose $-\alpha_3 < \beta_3 <\alpha_3.$

The numbers $\alpha$ and $\beta$ of the rational tangle $\beta/\alpha$ corresponds to the number of lines of the tangle, i.e. $\alpha$ means the number of horizontal lines and $\beta$ means the number of the vertical lines of the tangle.
Since we have $\alpha_2 \equiv \alpha_3 \equiv 1 \mod 2$, we let both $\alpha_2$ and $\alpha_3$ be positive odd numbers.
Then, we have the following three cases; 
\begin{itemize}
\item[(i)] both $\beta_2$ and $\beta_3$ are even,
\item[(ii)] either $\beta_2$ or $\beta_3$ is odd,
\item[(iii)] both $\beta_2$ and $\beta_3$ are odd.
\end{itemize}
From Lemma \ref{the continued fraction expansions}, the statement is true for the case (i).
Hence, we consider the case (ii) and (iii).

 \begin{figure}[h]
  \begin{center}
\includegraphics[clip,width=7cm]{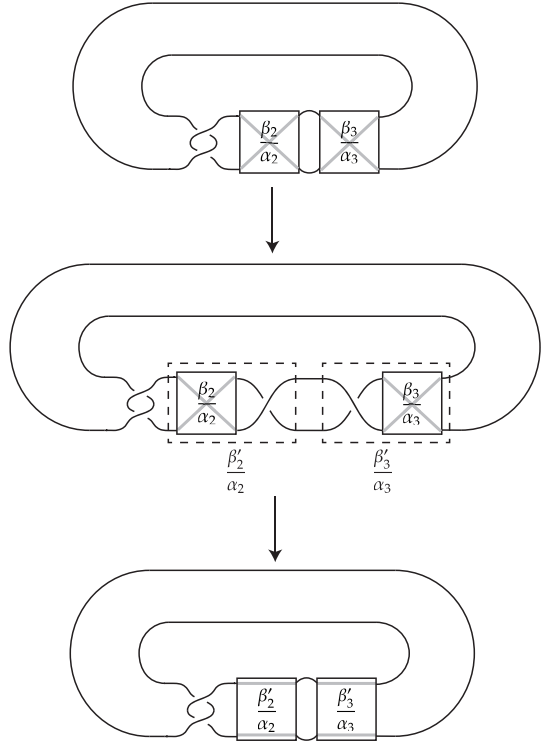}
\caption{The case of both $\beta_2$ and $\beta_3$ are odd}
  \end{center}
\end{figure}
In the case (iii), we add two half twists and denote boxes of dotted line depicted in Figure 4 by $\frac{\beta_2'}{\alpha_2}$ and $\frac{\beta_3'}{\alpha_3}$.
Then we have
\begin{align*}
\frac{\beta_2'}{\alpha_2}=& \frac{\beta_2}{\alpha_2} + \frac{|\beta_3|}{ \beta_3},\\
\frac{\beta_3'}{\alpha_3}=& \frac{\beta_3}{\alpha_3} - \frac{|\beta_3|}{ \beta_3},
\end{align*}
and hence both $\beta_2'$ and $\beta_3'$ are even.

\begin{figure}[h]
  \begin{center}
\includegraphics[clip,width=7cm]{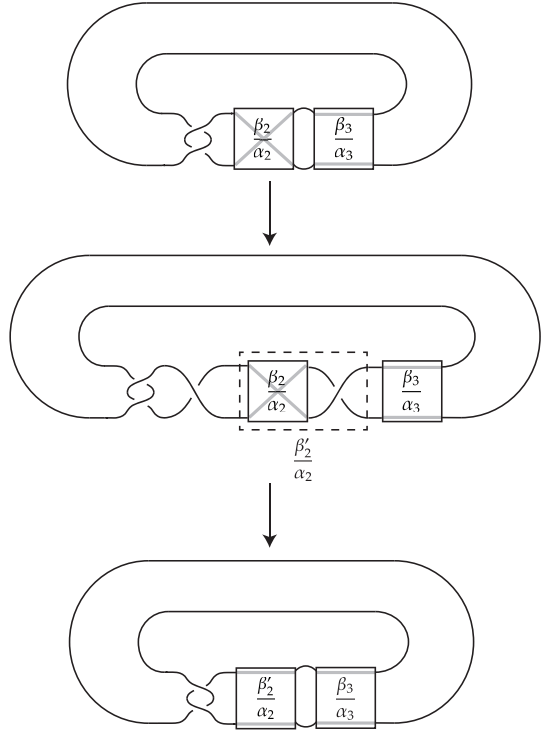}
\caption{The case of either $\beta_2$ or $\beta_3$ is odd}
  \end{center}
\end{figure}
In the case (ii), 
if $\beta_2$ is odd and $\beta_3$ is even, we add two half twists and denote boxes of dotted line depicted in Figure 5 by $\frac{\beta_2'}{\alpha_2}$.
Then we obtain new first tangle
\begin{align*}
\frac{\beta_1''}{\alpha_1}= -\frac{1}{2},
\end{align*}
and $\beta_2'$ is even
since
\begin{align*}
\frac{\beta_2'}{\alpha_2}=& \frac{\beta_2}{\alpha_2} + 1.
\end{align*}
Similary, if $\beta_2$ is even and $\beta_3$ is odd, we can obtain 
\begin{align*}
\frac{\beta_1''}{\alpha_1}= -\frac{1}{2},
\end{align*} 
and even integer $\beta_3'$.
\end{proof}

\subsection{Main theorem}

We consider the knot depicted in Figure 6. Note that two rational tangles are written by
\begin{align*}
\frac{\beta_2}{\alpha_2}=&\ 2 m_0 + \frac{1}{2 m_1+\displaystyle\frac{1}{2 m_2 +  \displaystyle\frac{}{ \ddots +\displaystyle\frac{1}{2 m_k}}}},\\
\frac{\beta_3}{\alpha_3}=&\ \frac{1}{2 n_1+\displaystyle\frac{1}{2 n_2 +  \displaystyle\frac{}{ \ddots +\displaystyle\frac{1}{2 n_l}}}},
\end{align*}
where both $k$ and $l$ are even. 
Then, we put
\begin{align*}
x_{-4} =& c^{-1} a^{-1} c, \ x_{-3} = a, \ x_{-2} = b,\\
y_{-2} =& b, \ y_{-1} = c, \ y_{0} = c^{-1} a c^{-1} a^{-1} c.
\end{align*}
As in the case (1), by extending these, we can take generators $\{x_{-4}, \cdots, x_{2k}, y_{-2}, \cdots, y_{2l}\}$ of the knot group $G(K)$ such that the following relations hold:
\begin{align*}
r_{2i-1}: \hspace{2mm} x_{2i-1} &= (x_{2i-3}^{(-1)^i} x_{2i-4}^{(-1)^i})^{m_i} x_{2i-3} (x_{2i-3}^{(-1)^i} x_{2i-4}^{(-1)^i})^{-m_i},\\
r_{2i}: \hspace{5.55mm} x_{2i} &= (x_{2i-3}^{(-1)^i} x_{2i-4}^{(-1)^i})^{m_i} x_{2i-4} (x_{2i-3}^{(-1)^i} x_{2i-4}^{(-1)^i})^{-m_i},\\
s_{2i-1}: \hspace{2.2mm} y_{2i-1} &= (y_{2i-3}^{(-1)^i} y_{2i-4}^{(-1)^i})^{-n_i} y_{2i-3} (y_{2i-3}^{(-1)^i} y_{2i-4}^{(-1)^i})^{n_i},\\
s_{2i}: \hspace{5.75mm} y_{2i} &= (y_{2i-3}^{(-1)^i} y_{2i-4}^{(-1)^i})^{-n_i} y_{2i-4} (y_{2i-3}^{(-1)^i} y_{2i-4}^{(-1)^i})^{n_i},
\end{align*}
and
\begin{align*} 
r_{2k+1}: \hspace{2mm} x_{2k-1} =&\  x_{2k-2}^{-1},\\
s_{2l+1}: \hspace{2.92mm} y_{2l-1} =&\ y_{2l-2}^{-1}, \\
rs: \hspace{5.61mm} x_{2k} =&\ y_{2l}.
\end{align*}

\begin{figure}[h]
  \begin{center}
\includegraphics[clip,width=6cm]{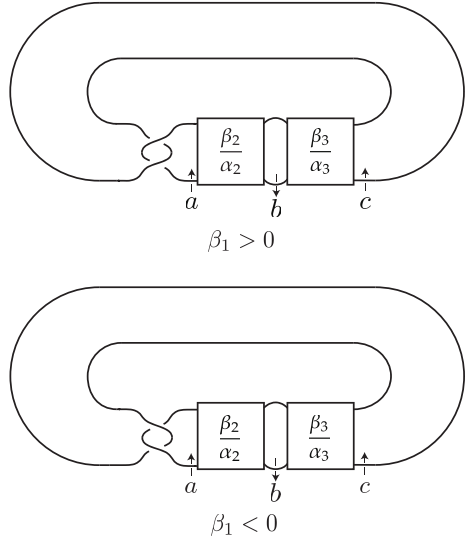}
\caption{The knot of case (2)}
  \end{center}
\end{figure}

Let $\rho : G(K) \to SL_2(\bC)$ be a representation of the knot group $G(K)$ and put
\[
\rho(a):=A, \ \rho(b):=B, \ \rho(c):=C, \ \rho(x_i):=X_i, \ \rho(y_i) :=Y_i.
\]
Then we have the following theorem.

\begin{thm}\label{the case 2}
We have
\begin{eqnarray*}
\Delta_{K,\rho}(t)
= 
  \begin{cases}
  \kappa_0 t^0 + \cdots  + \kappa_0 t^{2(k+l+1)} & $if$ \ m_0 \neq 0,\\
  \lambda_0 t^0 + \cdots + \lambda_0 t^{2(k+l)-2} & $if$ \ m_0 = 0,\\
  \end{cases}
\end{eqnarray*}
where $\kappa_0$ and $\lambda_0$ are given by the following:

 \begin{align*}
\kappa_0 =&\ \prod^k_{i=0} \left| \sum^{|m_i|}_{j=1} (X_{2i-3}^{(-1)^i} X_{2i-4}^{(-1)^i})^j \right|
\prod^l_{i=1} \left| \sum^{|n_i|}_{j=1} (Y_{2i-3}^{(-1)^i} Y_{2i-4}^{(-1)^i})^j \right|,\\
\lambda_0 =&\  \prod^k_{i=2} \left| \sum^{|m_i|}_{j=1} (X_{2i-3}^{(-1)^i} X_{2i-4}^{(-1)^i})^j \right|
\prod^l_{i=2} \left| \sum^{|n_i|}_{j=1} (Y_{2i-3}^{(-1)^i} Y_{2i-4}^{(-1)^i})^j \right| 
\lambda,
 \end{align*}
 where
\begin{align*}
\lambda=
 \begin{cases}
 \vspace{2mm}
\biggl| \frac{|m_1|}{m_1} \Bigl(\sum_j (B A)^j\Bigr) + \frac{|n_1|}{n_1}A^{-1} C \Bigl(\sum_j (B C)^j\Bigr) 
 + \frac{|m_1 n_1|}{m_1 n_1}\Bigl(\sum_j (B A)^j\Bigr) \Bigl(\sum_j (B C)^j\Bigr)BC \biggr| &  $if$ \ \beta_1 > 0,\\
\biggl| \frac{|m_1|}{m_1} \Bigl(\sum_j (B A)^j\Bigr) + \frac{|n_1|}{n_1}  A C^{-1} \Bigl(\sum_j (B C)^j\Bigr) 
 - \frac{|m_1 n_1|}{m_1 n_1}\Bigl(\sum_j (B A)^j\Bigr) \Bigl(\sum_j (B C)^j\Bigr) \biggr| &  $if$ \ \beta_1 < 0.
\end{cases}
\end{align*}
 \end{thm}

\subsection{Proof of Theorem \ref{the case 2}}
We have
\begin{eqnarray*} 
G(K) 
= \langle a,b,c, x_{-4}, \cdots, x_{2k}, y_{-2}, \cdots, y_{2l}  \ | \  r_{-4}, \cdots, r_{2k+1}, s^{-2}, \cdots, s^{2l+1} \rangle.
\end{eqnarray*}
Suppose $\Phi=(\rho \otimes \frak{a}) \circ \phi$.
Then we put 
\begin{align*}
R_{-4}=&\ \Phi \left( \frac{\partial r_{-4}}{\partial a} \right) 
\end{align*}
and 
\begin{align*}
R_{2i-1}=&\ \Phi \left( \frac{\partial r_{2i-1}}{\partial x_{2i-4}} \right), \
R'_{2i-1}= \Phi \left( \frac{\partial r_{2i-1}}{\partial x_{2i-3}} \right), \\
R_{2i}=&\ \Phi \left( \frac{\partial r_{2i}}{\partial x_{2i-4}} \right), \ \ \ \ 
R'_{2i}= \Phi \left( \frac{\partial r_{2i}}{\partial x_{2i-3}} \right) ,
\end{align*}
for $0 \leq i \leq k$,
\begin{align*}
S_{2i-1}=&\ \Phi \left( \frac{\partial s_{2i-1}}{\partial y_{2i-4}} \right), \
S'_{2i-1}= \Phi \left( \frac{\partial s_{2i-1}}{\partial y_{2i-3}} \right), \\
S_{2i}=&\ \Phi \left( \frac{\partial s_{2i}}{\partial y_{2i-4}} \right), \ \ \ \ 
S'_{2i}= \Phi \left( \frac{\partial s_{2i}}{\partial y_{2i-3}} \right) ,
\end{align*}
for $1 \leq i \leq l$.

\begin{prop}\label{TAP of the case 2}
We have
\[
\displaystyle \Delta_{K,\rho}(t) =
 \frac{\left|
\begin{array}{cc}
M_{2k+1} & M_{2k+1}'\\
N_{2l+1} & N_{2l+1}'
\end{array}
\right|}{t^2 - \mathrm{tr}  \rho(c) t +1},
\]
where the matrices $M_{2k+1}, M_{2k+1}', N_{2l+1} , N_{2l+1}'$ are computed as in the case {\rm (1)}.
\end{prop}

\begin{proof}
For the generator $c$, let $A_{\rho,c}$ be the matrix obtained by removing the 2 columns corresponding to $c$ from the twisted Alexander matrix $A_{\rho}$.
Then, $|A_{\rho,c}|$ is given by

\begin{eqnarray*}
\hspace{-3mm}&& \hspace{-3mm} 
\left|
\begin{array}{ccccccccccccccccccc}
R_{-4} & \hspace{-1mm}O & \hspace{-1mm} E & \hspace{-1mm} O & \hspace{-3mm}\cdots & \hspace{-3mm} O & \hspace{-1mm} O & \hspace{-2mm} O & \hspace{-2mm} O & O & O & O & \hspace{-3mm} \cdots & \hspace{-3mm} O & \hspace{-2mm} O & \hspace{-2mm} O & \hspace{-2mm} O & O\\ 

-E & \hspace{-1mm}O & \hspace{-1mm} O & \hspace{-1mm} E & \hspace{-3mm} & \hspace{-3mm} O & \hspace{-1mm} O & \hspace{-2mm} O & \hspace{-2mm} O & O & O & O & \hspace{-3mm} \cdots & \hspace{-3mm} O & \hspace{-2mm} O & \hspace{-2mm} O & \hspace{-2mm} O & O\\

O & \hspace{-1mm}-E & \hspace{-1mm} O & \hspace{-1mm} O & \hspace{-3mm}\ddots & \hspace{-3mm} O & \hspace{-1mm} O & \hspace{-2mm} O & \hspace{-2mm} O & O & O & O & \hspace{-3mm} \cdots & \hspace{-3mm} O & \hspace{-2mm} O & \hspace{-2mm} O & \hspace{-2mm} O & O\\

O & \hspace{-1mm}O & \hspace{-1mm} R_{-1} & \hspace{-1mm} R_{-1}' & \hspace{-3mm}& \hspace{-3mm} O & \hspace{-1mm} O & \hspace{-2mm} O & \hspace{-2mm} O & O & O & O & \hspace{-3mm} \cdots & \hspace{-3mm} O & \hspace{-2mm} O & \hspace{-2mm} O & \hspace{-2mm} O & O\\

O & \hspace{-1mm}O & \hspace{-1mm} R_{0} & \hspace{-1mm} R_{0}' &\hspace{-3mm} & \hspace{-3mm} O & \hspace{-1mm} O & \hspace{-2mm} O & \hspace{-2mm} O & O & O & O & \hspace{-3mm} \cdots & \hspace{-3mm} O & \hspace{-2mm} O & \hspace{-2mm} O & \hspace{-2mm} O & O\\

\vdots & \hspace{-1mm}\vdots & \hspace{-1mm} & \hspace{-1mm} & \hspace{-3mm}\ddots & \hspace{-3mm} & \hspace{-1mm}  & \hspace{-2mm} \ddots & \hspace{-2mm} & \vdots & \vdots & \vdots & \hspace{-3mm} & \hspace{-3mm} \vdots & \hspace{-2mm} \vdots & \hspace{-2mm} \vdots & \hspace{-2mm} \vdots & \vdots\\

O & \hspace{-1mm}O & \hspace{-1mm} O & \hspace{-1mm} O & \hspace{-3mm} & \hspace{-3mm} R_{2k-1} & \hspace{-1mm} R_{2k-1}' & \hspace{-2mm} O & \hspace{-2mm} E & O & O & O & \hspace{-3mm} \cdots & \hspace{-3mm} O & \hspace{-2mm} O & \hspace{-2mm} O & \hspace{-2mm} O & O\\ 

O & \hspace{-1mm}O & \hspace{-1mm} O & \hspace{-1mm} O & \hspace{-3mm} & \hspace{-3mm} R_{2k} & \hspace{-1mm} R_{2k}' & \hspace{-2mm} O & \hspace{-2mm} O & E & O & O & \hspace{-3mm} \cdots & \hspace{-3mm} O & \hspace{-2mm} O & \hspace{-2mm} O & \hspace{-2mm} O & O\\ 

O & \hspace{-1mm}O & \hspace{-1mm} O & \hspace{-1mm} O & \hspace{-3mm}\cdots & \hspace{-3mm} O & \hspace{-1mm} O & \hspace{-2mm} t X_{2k-1} & \hspace{-2mm}E & O & O & O & \hspace{-3mm} \cdots & \hspace{-3mm} O & \hspace{-2mm} O & \hspace{-2mm} O & \hspace{-2mm} O & O\\ 

O & \hspace{-1mm}-E & \hspace{-1mm} O & \hspace{-1mm} O & \hspace{-3mm}\cdots & \hspace{-3mm} O & \hspace{-1mm} O & \hspace{-2mm} O & \hspace{-2mm} O & O & E & O & \hspace{-3mm} \cdots & \hspace{-3mm} O & \hspace{-2mm} O & \hspace{-2mm} O & \hspace{-2mm} O & O\\ 

O & \hspace{-1mm}O & \hspace{-1mm} O & \hspace{-1mm} O & \hspace{-3mm}\cdots & \hspace{-3mm} O & \hspace{-1mm} O & \hspace{-2mm} O & \hspace{-2mm} O & O & O & E & \hspace{-3mm} & \hspace{-3mm} O & \hspace{-2mm} O & \hspace{-2mm} O & \hspace{-2mm} O & O\\ 

S_0 & \hspace{-1mm}O & \hspace{-1mm} O & \hspace{-1mm} O & \hspace{-3mm}\cdots & \hspace{-3mm} O & \hspace{-1mm} O & \hspace{-2mm} O & \hspace{-2mm} O & O & O & O & \hspace{-3mm} \ddots & \hspace{-3mm} O & \hspace{-2mm} O & \hspace{-2mm} O & \hspace{-2mm} O & O\\ 

O & \hspace{-1mm}O & \hspace{-1mm} O & \hspace{-1mm} O & \hspace{-3mm}\cdots & \hspace{-3mm} O & \hspace{-1mm} O & \hspace{-2mm} O & \hspace{-2mm} O & O & S_1 & S_1' & \hspace{-3mm}  & \hspace{-3mm} O & \hspace{-2mm} O & \hspace{-2mm} O & \hspace{-2mm} O & O\\ 

O & \hspace{-1mm}O & \hspace{-1mm} O & \hspace{-1mm} O & \hspace{-3mm}\cdots & \hspace{-3mm} O & \hspace{-1mm} O & \hspace{-2mm} O & \hspace{-2mm} O & O & S_2 & S_2' & \hspace{-3mm} & \hspace{-3mm} O & \hspace{-2mm} O & \hspace{-2mm} O & \hspace{-2mm} O & O\\ 

\vdots & \hspace{-1mm}\vdots & \hspace{-1mm} \vdots & \hspace{-1mm} \vdots & \hspace{-3mm} & \hspace{-3mm} \vdots & \hspace{-1mm} \vdots & \hspace{-2mm} \vdots & \hspace{-2mm} \vdots & \vdots &  &  & \hspace{-3mm} \ddots & \hspace{-3mm} & \hspace{-2mm} & \hspace{-2mm} \ddots &\hspace{-2mm}  & \vdots\\

O & \hspace{-1mm}O & \hspace{-1mm} O & \hspace{-1mm} O & \hspace{-3mm}\cdots & \hspace{-3mm} O & \hspace{-1mm} O & \hspace{-2mm} O & \hspace{-2mm} O & O & O & O & \hspace{-3mm} & \hspace{-3mm} S_{2l-1} & \hspace{-2mm} S_{2l-1}' & \hspace{-2mm} O & \hspace{-2mm} E & O\\ 

O & \hspace{-1mm}O & \hspace{-1mm} O & \hspace{-1mm} O & \hspace{-3mm}\cdots & \hspace{-3mm} O & \hspace{-1mm} O & \hspace{-2mm} O & \hspace{-2mm} O & O & O & O & \hspace{-3mm} & \hspace{-3mm} S_{2l} & \hspace{-2mm} S_{2l}' & \hspace{-2mm} O & \hspace{-2mm} O & E\\ 

O & \hspace{-1mm}O & \hspace{-1mm} O & \hspace{-1mm} O & \hspace{-3mm}\cdots & \hspace{-3mm} O & \hspace{-1mm} O & \hspace{-2mm} O & \hspace{-2mm} O & O & O & O & \hspace{-3mm} \cdots & \hspace{-3mm} O & \hspace{-2mm} O & \hspace{-2mm} t Y_{2l-1} & \hspace{-2mm} E & O
\end{array}
\right| \\
\hspace{-3mm}&=& \hspace{-3mm} 
\left|
\begin{array}{ccccccccccccccccccc}
M_{-4} & M_{-4}' & E & O & \cdots & O & O & O & O & O & O & O & \cdots & O & O & O & O & O\\ 

M_{-3} & M_{-3}' & O & E & & O & O & O & O & O & O & O & \cdots & O & O & O & O & O\\

M_{-2} & M_{-2}' & O & O & \ddots & O & O & O & O & O & O & O & \cdots & O & O & O & O & O\\

M_{-1} & M_{-1}' & O & O & & O & O & O & O & O & O & O & \cdots & O & O & O & O & O\\

M_{0} & M_{0}' & O & O & & O & O & O & O & O & O & O & \cdots & O & O & O & O & O\\

\vdots & \vdots & & & & & & \ddots & & \vdots & \vdots & \vdots & & \vdots & \vdots & \vdots & \vdots & \vdots\\

M_{2k-1} & M_{2k-1}' & O & O & \cdots & O & O & O & E & O & O & O & \cdots & O & O & O & O & O\\ 

M_{2k} & M_{2k}' & O & O & \cdots & O & O & O & O & E & O & O & \cdots & O & O & O & O & O\\ 

M_{2k+1} & M_{2k+1}'  & O & O & \cdots & O & O & O & O & O & O & O & \cdots & O & O & O & O & O\\ 

N_{-2} & N_{-2}' & O & O & \cdots & O & O & O & O & O & E & O & \cdots & O & O & O & O & O\\ 

N_{-1} & N_{-1}' & O & O & \cdots & O & O & O & O & O & O & E &  & O & O & O & O & O\\ 

N_{0} & N_{0}' & O & O & \cdots & O & O & O & O & O & O & O & \ddots & O & O & O & O & O\\ 

N_{1} & N_{1}' & O & O & \cdots & O & O & O & O & O & O & O & & O & O & O & O & O\\ 

N_{2} & N_{2}' & O & O & \cdots & O & O & O & O & O & O & O & & O & O & O & O & O\\ 

\vdots & \vdots & \vdots & \vdots &  & \vdots & \vdots & \vdots & \vdots & \vdots & \vdots & \vdots & & & & \ddots & & \vdots\\

N_{2l-1} & N_{2l-1}' & O & O & \cdots & O & O & O & O & O & O & O & \cdots & O & O & O & E & O\\ 

N_{2l} & N_{2l}' & O & O & \cdots & O & O & O & O & O & O & O & \cdots & O & O & O & O & E\\ 

N_{2l+1} & N_{2l+1}' & O & O & \cdots & O & O & O & O & O & O & O & \cdots & O & O & O & O& O
\end{array}
\right| \\
\hspace{-3mm}&=& \hspace{-3mm} 
\left|
\begin{array}{cc}
M_{2k+1} & M_{2k+1}'\\
N_{2l+1} & N_{2l+1}'
\end{array}
\right|,
\end{eqnarray*}
where
the matrix $M_{2k+1}$, $M_{2k+1}'$, $N_{2l+1}$ and $N_{2l+1}$ are defined as follows:
We consider the sequences $\{M_i\}^{2k+1}_{-4}$ and $\{M_i'\}^{2k+1}_{-4}$ which are defined by
 \begin{align*}
M_{-4} =&\ R_{-4},\
M_{-3} = -E,\ 
M_{-2} = O,\\
M'_{-4} =&\ O,\ \ \ \ 
M'_{-3} = O,\ \ \ \! \
M'_{-2} = -E,
 \end{align*}
 and
\begin{align*}
M_{2i-1} =&\ (-R_{2i-1} )M_{2i-4} + (-R'_{2i-1}) M_{2i-3},\\
M_{2i} =&\ (-R_{2i} )M_{2i-4} + (-R'_{2i}) M_{2i-3},\\
M_{2i-1}' =&\ (-R_{2i-1} )M_{2i-4}' + (-R'_{2i-1}) M_{2i-3}',\\
M_{2i}' =&\ (-R_{2i} )M_{2i-4}' + (-R'_{2i}) M_{2i-3}',
\end{align*}
for $0 \le i \le k$. Then we put
\begin{align*}
M_{2k+1}=&\ (-t X_{2k-1}) M_{2k-2} + (-E) M_{2k-1},\\
M_{2k+1}'=&\ (-t X_{2k-1}) M_{2k-2}' + (-E) M_{2k-1}'.
\end{align*}
Similarly, we consider the sequences $\{N_i\}^{2l+1}_{-2}$ and $\{N_i'\}^{2l+1}_{-2}$ which are defined by 
 \begin{align*}
N_{-2} =&\ O,\ \ \ \ \!
N_{-1} = O,\
N_{0} = S_0,\\
N'_{-2} =& -E,\
N'_{-1} = O,\
N'_{0} = O,
 \end{align*}
 and
  \begin{align*}
N_{2i-1} =&\ (-S_{2i-1} )N_{2i-4} + (-S'_{2i-1}) N_{2i-3},\\
N_{2i} =&\ (-S_{2i} )N_{2i-4} + (-S'_{2i}) N_{2i-3},\\
N_{2i-1}' =&\ (-S_{2i-1} )N_{2i-4}' + (-S'_{2i-1}) N_{2i-3}',\\
N_{2i} =&\ (-S_{2i} )N_{2i-4}' + (-S'_{2i}) N_{2i-3}',
 \end{align*}
for $1 \le i \le l$. Then we set
\begin{align*}
N_{2l+1} =&\ (-t Y_{2l-1}) N_{2l-2} + (-E) N_{2l-1},\\
N_{2l+1}' =&\ (-t Y_{2l-1}) N_{2l-2}' + (-E) N_{2l-1}'.
 \end{align*}
\end{proof}

\begin{lem}
If $m_0 \neq 0$, then the matrix $\left|
\begin{array}{cc}
M_{2k+1} & M_{2k+1}'\\
N_{2l+1} & N_{2l+1}'
\end{array}
\right|$
is given by
\begin{eqnarray*}
\left|
\begin{array}{cc}
M_{2k+1} & M_{2k+1}'\\
N_{2l+1} & N_{2l+1}'
\end{array}
\right|
=t^{-(k+l)} |M_{2k+1}^{\min}| |N_{2l+1}^{' \min}| + \cdots + t^{k+ l + 4} |M_{2k+1}^{\max}| |N_{2l+1}^{' \max}| ,
 \end{eqnarray*}
 where
 \begin{align*}
|M_{2k+1}^{\max}| =|M_{2k+1}^{\min}| =& \prod^k_{i=0} \left| \sum^{|m_i|}_{j=1} (X_{2i-3}^{(-1)^i} X_{2i-4}^{(-1)^i})^j \right|,\\
|N_{2l+1}^{' \max}| = |N_{2l+1}^{' \min}| =& \prod^l_{i=1} \left| \sum^{|n_i|}_{j=1} (Y_{2i-3}^{(-1)^i} Y_{2i-4}^{(-1)^i})^j \right|.
 \end{align*} 
\end{lem}

\begin{proof}
The matrices $M_{2k+1}$ and $M'_{2k+1}$ are computed as in the case (1), that is, we have
\begin{align*}
M_{2k+1}=& t^{-\frac{k}{2}-1} M_{2k+1}^{-\frac{k}{2}-1} + \cdots + t^{\frac{k}{2}+1} M_{2k+1}^{\frac{k}{2}+1},\\
M_{2k+1}'=&  t^{-\frac{k}{2}+1} {M'}_{2k+1}^{-\frac{k}{2}+1} + \cdots + t^{\frac{k}{2}+1} {M'}_{2k+1}^{\frac{k}{2}+1}.
\end{align*}
Similarly, we obtain 
 \begin{align*}
N_{2l+1} =& t^{-\frac{l}{2}} N_{2l+1}^{-\frac{l}{2}} + \cdots + t^{\frac{l}{2}} N_{2l+1}^{\frac{l}{2}},\\
N_{2l+1}' =& t^{-\frac{l}{2}+1} {N'}_{2l+1}^{-\frac{l}{2}+1} + \cdots + t^{\frac{l}{2}+1} {N'}_{2l+1}^{\frac{l}{2}+1}.
 \end{align*}
Hence, we have
\begin{align*} 
\left|
\begin{array}{cc}
M_{2k+1} & M_{2k+1}'\\
N_{2l+1} & N_{2l+1}'
\end{array}
\right| 
=&
\left|
\begin{array}{cc}
t^{-\frac{k}{2}-1} M_{2k+1}^{-\frac{k}{2}-1} + \cdots + t^{\frac{k}{2}+1} M_{2k+1}^{\frac{k}{2}+1} & t^{-\frac{k}{2}+1} {M'}_{2k+1}^{-\frac{k}{2}+1} + \cdots + t^{\frac{k}{2}+1} {M'}_{2k+1}^{\frac{k}{2}+1}\\
t^{-\frac{l}{2}} N_{2l+1}^{-\frac{l}{2}} + \cdots + t^{\frac{l}{2}} N_{2l+1}^{\frac{l}{2}} & t^{-\frac{l}{2}+1} {N'}_{2l+1}^{-\frac{l}{2}+1} + \cdots + t^{\frac{l}{2}+1} {N'}_{2l+1}^{\frac{l}{2}+1}
\end{array}
\right| \\
=& t^{-(k+l)} \left|M_{2k+1}^{-\frac{k}{2}-1} \right|  \left| {N'}_{2l+1}^{-\frac{l}{2}+1} \right| + \cdots + t^{k+l+4}  \left|M_{2k+1}^{\frac{k}{2}+1} \right|  \left|{N'}_{2l+1}^{\frac{l}{2}+1} \right|.
\end{align*} 
Then, since we have
\begin{align*}
M_{2k+1}^{\max}:=& M_{2k+1}^{\frac{k}{2}+1}\\ 
=& (-1)^{\frac{k}{2}+1}  \frac{|m_k m_{k-1}\cdots m_1|}{m_k m_{k-1} \cdots m_1}  X_{2k-1}\\
& \Bigl\{  \sum_{j}^{} (X_{2k-3} X_{2k-4})^j X_{2k-4}^{-1} \sum_{j}^{} (X_{2k-5}^{-1} X_{2k-6}^{-1})^j \Bigr\} \cdots 
\Bigl\{\sum_{j}^{} (X_{1} X_{0})^j  X_{0}^{-1}  \sum_{j}^{} (X_{-1}^{-1} X_{-2}^{-1})^j \Bigr\} X_{-1}^{-1} M_{-1}^{\max},\\
M_{2k+1}^{\min}:=& M_{2k+1}^{-\frac{k}{2}-1} \\
=& (-1)^{\frac{k}{2}+1} \frac{|m_k m_{k-1}\cdots m_1|}{m_k m_{k-1} \cdots m_1} \sum_{j}^{} (X_{2k-3} X_{2k-4})^j  \\
& \Bigl\{ \sum_{j}^{} (X_{2k-5}^{-1} X_{2k-6}^{-1})^j X_{2k-6} \sum_{j}^{} (X_{2k-7} X_{2k-8})^j\Bigr\} \cdots 
\Bigl\{ \sum_{j}^{} (X_{3}^{-1} X_{2}^{-1})^j  X_2\sum_{j}^{} (X_{1} X_{0})^j  X_{0}^{-1}  \Bigr\}\\
& \sum_{j}^{} (X_{-1}^{-1} X_{-2}^{-1})^j X_{-2} M_{-1}^{\min},
\end{align*}
where
\begin{align*}
M_{-1}^{\max}=& 
\begin{cases}
\frac{|m_0|}{m_0} X_{-1} \sum_{j}^{} (X_{-3} X_{-4})^j X_{-1}^{-1} A^{-1} &  $if$ \ \beta_1 > 0,\\
-\frac{|m_0|}{m_0} X_{-1} \sum_{j}^{} (X_{-3} X_{-4})^j X_{-1}^{-1} A^{-1} C A^{-1} &  $if$ \ \beta_1 < 0,
\end{cases}\\
M_{-1}^{\min}=& 
\begin{cases}
\frac{|m_0|}{m_0} \sum_{j}^{} (X_{-3} X_{-4})^j C^{-1} &  $if$ \ \beta_1 > 0,\\
-\frac{|m_0|}{m_0} \sum_{j}^{} (X_{-3} X_{-4})^j A^{-1} &  $if$ \ \beta_1 < 0,
\end{cases}
\end{align*}
we get
 \begin{align*}
|M_{2k+1}^{\max}| =|M_{2k+1}^{\min}| =& \prod^k_{i=0} \left| \sum^{|m_i|}_{j=1} (X_{2i-3}^{(-1)^i} X_{2i-4}^{(-1)^i})^j \right|. 
\end{align*}

Similarly, since we have
\begin{align*}
{N'}_{2l+1}^{\max}:=& {N'}_{2l+1}^{\frac{l}{2}+1}\\ 
=& (-1)^{\frac{l}{2}}  \frac{|n_l n_{l-1}\cdots n_1|}{n_l n_{l-1} \cdots n_1}  Y_{2l-1}\\
& \Bigl\{  \sum_{j}^{} (Y_{2l-3} Y_{2l-4})^j Y_{2l-4}^{-1} \sum_{j}^{} (Y_{2l-5}^{-1} Y_{2l-6}^{-1})^j \Bigr\} \cdots 
\Bigl\{\sum_{j}^{} (Y_{1} Y_{0})^j  Y_{0}^{-1}  \sum_{j}^{} (Y_{-1}^{-1} Y_{-2}^{-1})^j \Bigr\} ,\\
{N'}_{2l+1}^{\min}:=& {N'}_{2l+1}^{-\frac{l}{2}+1} \\
=& (-1)^{\frac{l}{2}}  \frac{|n_l n_{l-1}\cdots n_1|}{n_l n_{l-1} \cdots n_1} \sum_{j}^{} (Y_{2l-3} Y_{2l-4})^j  \\
& \Bigl\{ \sum_{j}^{} (Y_{2l-5}^{-1} Y_{2l-6}^{-1})^j Y_{2l-6} \sum_{j}^{} (Y_{2l-7} Y_{2l-8})^j\Bigr\} \cdots 
\Bigl\{ \sum_{j}^{} (Y_{3}^{-1} Y_{2}^{-1})^j  Y_2\sum_{j}^{} (Y_{1} Y_{0})^j Y_{0}^{-1}  \Bigr\}\\
& \sum_{j}^{} (Y_{-1}^{-1} Y_{-2}^{-1})^j ,
\end{align*}
we obtain
 \begin{align*}
|N_{2l+1}^{' \max}| = |N_{2l+1}^{' \min}| =& \prod^l_{i=1} \left| \sum^{|n_i|}_{j=1} (Y_{2i-3}^{(-1)^i} Y_{2i-4}^{(-1)^i})^j \right|.
 \end{align*} 

\end{proof}

\begin{lem}
If $m_0 = 0$, then
$\left|
\begin{array}{cc}
M_{2k+1} & M_{2k+1}'\\
N_{2l+1} & N_{2l+1}'
\end{array}
\right|$
is given by
\begin{eqnarray*}
 t^{-(k+l)+2} \left|
\begin{array}{cc}
\hspace{-1mm} M_{2k+1}^{\min} & \hspace{-2mm} M_{2k+1}^{' \min} \hspace{-1mm} \\
\hspace{-1mm} N_{2l+1}^{\min} & \hspace{-2mm} N_{2l+1}^{' \min} \hspace{-1mm} 
\end{array}
\right|
 + \cdots + 
 t^{k+ l + 2} 
 \left|
\begin{array}{cc}
\hspace{-1mm} M_{2k+1}^{\max} & \hspace{-2mm} M_{2k+1}^{' \max} \hspace{-1mm} \\
\hspace{-1mm} N_{2l+1}^{\max} & \hspace{-2mm} N_{2l+1}^{' \max} \hspace{-1mm} 
\end{array}
\right|
 \end{eqnarray*}
 where
 $\left|
\begin{array}{cc}
\hspace{-1mm} M_{2k+1}^{\max} & \hspace{-2mm} M_{2k+1}^{' \max} \hspace{-1mm} \\
\hspace{-1mm} N_{2l+1}^{\max} & \hspace{-2mm} N_{2l+1}^{' \max} \hspace{-1mm} 
\end{array}
\right|$
and 
$\left|
\begin{array}{cc}
\hspace{-1mm} M_{2k+1}^{\min} & \hspace{-2mm} M_{2k+1}^{' \min} \hspace{-1mm} \\
\hspace{-1mm} N_{2l+1}^{\min} & \hspace{-2mm} N_{2l+1}^{' \min} \hspace{-1mm} 
\end{array}
\right|$ are given by
 \begin{align*}
\prod^k_{i=2} \left| \sum^{|m_i|}_{j=1} (X_{2i-3}^{(-1)^i} X_{2i-4}^{(-1)^i})^j \right|
\prod^l_{i=2} \left| \sum^{|n_i|}_{j=1} (Y_{2i-3}^{(-1)^i} Y_{2i-4}^{(-1)^i})^j \right| 
\lambda,
 \end{align*} 
 and
 \begin{align*}
\lambda=
 \begin{cases}
 \vspace{2mm}
\biggl| \frac{|m_1|}{m_1} \Bigl(\sum_j (B A)^j\Bigr) + \frac{|n_1|}{n_1}A^{-1} C \Bigl(\sum_j (B C)^j\Bigr) 
 + \frac{|m_1 n_1|}{m_1 n_1}\Bigl(\sum_j (B A)^j\Bigr) \Bigl(\sum_j (B C)^j\Bigr)BC \biggr| &  $if$ \ \beta_1 > 0,\\
\biggl| \frac{|m_1|}{m_1} \Bigl(\sum_j (B A)^j\Bigr) + \frac{|n_1|}{n_1}  A C^{-1} \Bigl(\sum_j (B C)^j\Bigr) 
 - \frac{|m_1 n_1|}{m_1 n_1}\Bigl(\sum_j (B A)^j\Bigr) \Bigl(\sum_j (B C)^j\Bigr) \biggr| &  $if$ \ \beta_1 < 0.
\end{cases}
\end{align*}
\end{lem}

\begin{proof}
If $m_0=0$, since we have
 \begin{align*}
 x_{-1}=x_{-3},\\
 x_0= x_{-4},
 \end{align*}
then we get
 \begin{align*}
 R_{-1}=O,\
 R_{-1}' = -E,\
 R_0 = -E,\
 R_0' = O.
 \end{align*}
 Hence, we can compute as in the case of $m_0 \neq 0$ and obtain 
\begin{align*}
M_{2k+1} =& t^{-\frac{k}{2}} M_{2k+1}^{-\frac{k}{2}} + \cdots + t^{\frac{k}{2}} M_{2k+1}^{\frac{k}{2}},\\
M_{2k+1}'=&  t^{-\frac{k}{2}+1} {M'}_{2k+1}^{-\frac{k}{2}+1} + \cdots + t^{\frac{k}{2}+1} {M'}_{2k+1}^{\frac{k}{2}+1}.
\end{align*}
Then, since $N_{2l+1}$ and $N'_{2l+1}$ are same as in the case of $m_0 \neq 0$, we have
\begin{align*} 
\left|
\begin{array}{cc}
M_{2k+1} & M_{2k+1}'\\
N_{2l+1} & N_{2l+1}'
\end{array}
\right| 
=&
\left|
\begin{array}{cc}
t^{-\frac{k}{2}} M_{2k+1}^{-\frac{k}{2}} + \cdots + t^{\frac{k}{2}} M_{2k+1}^{\frac{k}{2}} & t^{-\frac{k}{2}+1} {M'}_{2k+1}^{-\frac{k}{2}+1} + \cdots + t^{\frac{k}{2}+1} {M'}_{2k+1}^{\frac{k}{2}+1}\\
t^{-\frac{l}{2}} N_{2l+1}^{-\frac{l}{2}} + \cdots + t^{\frac{l}{2}} N_{2l+1}^{\frac{l}{2}} & t^{-\frac{l}{2}+1} {N'}_{2l+1}^{-\frac{l}{2}+1} + \cdots + t^{\frac{l}{2}+1} {N'}_{2l+1}^{\frac{l}{2}+1}
\end{array}
\right| \\
=& t^{-(k+l)+2} \left|
\begin{array}{cc}
M_{2k+1}^{-\frac{k}{2}} & {M'}_{2k+1}^{-\frac{k}{2}+1}\\
N_{2l+1}^{-\frac{l}{2}} & {N'}_{2l+1}^{-\frac{l}{2}+1}
\end{array}
\right| 
 + \cdots 
 + t^{k+l+2}  \left|
\begin{array}{cc}
M_{2k+1}^{\frac{k}{2}} & {M'}_{2k+1}^{\frac{k}{2}+1}\\
N_{2l+1}^{\frac{l}{2}} & {N'}_{2l+1}^{\frac{l}{2}+1}
\end{array}
\right| .
\end{align*} 
If we put
\begin{align*} 
M=& (-1)^{\frac{k}{2}}  \frac{|m_k m_{k-1}\cdots m_2|}{m_k m_{k-1} \cdots m_2}  X_{2k-1}\\
& \Bigl\{  \sum_{j}^{} (X_{2k-3} X_{2k-4})^j X_{2k-4}^{-1} \sum_{j}^{} (X_{2k-5}^{-1} X_{2k-6}^{-1})^j \Bigr\} \cdots 
\Bigl\{\sum_{j}^{} (X_{5} X_{4})^j  X_{4}^{-1}  \sum_{j}^{} (X_{3}^{-1} X_{2}^{-1})^j \Bigr\}\sum_{j}^{} (X_{1} X_{0})^j  X_{0}^{-1},\\
N=& (-1)^{\frac{l}{2}}  \frac{|n_l n_{l-1}\cdots n_2|}{n_l n_{l-1} \cdots n_2}  Y_{2l-1}\\
& \Bigl\{  \sum_{j}^{} (Y_{2l-3} Y_{2l-4})^j Y_{2l-4}^{-1} \sum_{j}^{} (Y_{2l-5}^{-1} Y_{2l-6}^{-1})^j \Bigr\} \cdots 
\Bigl\{\sum_{j}^{} (Y_{5} Y_{4})^j  Y_{4}^{-1}  \sum_{j}^{} (Y_{3}^{-1} Y_{2}^{-1})^j \Bigr\} \sum_{j}^{} (Y_{1} Y_{0})^j  Y_{0}^{-1},
\end{align*} 
then we have
\begin{align*}
\left|
\begin{array}{cc}
M_{2k+1}^{\frac{k}{2}} & {M'}_{2k+1}^{\frac{k}{2}+1}\\
N_{2l+1}^{\frac{l}{2}} & {N'}_{2l+1}^{\frac{l}{2}+1}
\end{array}
\right| 
=&
\left|
\begin{array}{cc}
-M (M_0^{-1} -X_1^{-1} R_1^0) & \frac{|m_1|}{m_1} M  \sum_{j}^{} (X_{-1}^{-1} X_{-2}^{-1})^j  \\
N S_0^{-1} & \frac{|n_1|}{n_1} N  \sum_{j}^{} (Y_{-1}^{-1} Y_{-2}^{-1})^j 
\end{array}
\right| \\
=&
\left|
\begin{array}{cc}
M & O \\
O & N
\end{array}
\right| 
\left|
\begin{array}{cc}
-(M_0^{-1} -X_1^{-1} R_1^0) & \frac{|m_1|}{m_1}\sum_{j}^{} (X_{-1}^{-1} X_{-2}^{-1})^j  \\
 S_0^{-1} & \frac{|n_1|}{n_1} \sum_{j}^{} (Y_{-1}^{-1} Y_{-2}^{-1})^j 
\end{array}
\right| \\
=&
|M| |N| 
\left|
\begin{array}{cc}
-(M_0^{-1} -X_1^{-1} R_1^0) & \frac{|m_1|}{m_1}\sum_{j}^{} (X_{-1}^{-1} X_{-2}^{-1})^j  \\
 S_0^{-1} & \frac{|n_1|}{n_1} \sum_{j}^{} (Y_{-1}^{-1} Y_{-2}^{-1})^j 
\end{array}
\right|.
\end{align*} 
Since we have
\begin{align*}
-(M_0^{-1} -X_1^{-1} R_1^0)=& 
\begin{cases}
\bigl\{(X_{-1}^{-1} X_{-2}^{-1})^{m_1+1}+ \frac{|m_1|}{m_1} \sum_{j}^{} (X_{-1}^{-1} X_{-2}^{-1})^j \bigr\} X_{-2} 
&  $if$ \ \beta_1 > 0,\\
\bigl\{-A^{-1} C + \frac{|m_1|}{m_1} \sum_{j}^{} (X_{-1}^{-1} X_{-2}^{-1})^j \bigr\} X_{-1}^{-1}  &  $if$ \ \beta_1 < 0,
\end{cases}\\
S_0^{-1}=& 
\begin{cases}
C^{-1} &  $if$ \ \beta_1 > 0,\\
A^{-1} &  $if$ \ \beta_1 < 0,
\end{cases}
\end{align*}
we obtain
\begin{align*}
& \left|
\begin{array}{cc}
-(M_0^{-1} -X_1^{-1} R_1^0) & \frac{|m_1|}{m_1}\sum_{j}^{} (X_{-1}^{-1} X_{-2}^{-1})^j  \\
 S_0^{-1} & \frac{|n_1|}{n_1} \sum_{j}^{} (Y_{-1}^{-1} Y_{-2}^{-1})^j 
\end{array}
\right|\\
=&
\begin{cases}
 \vspace{2mm}
\biggl| \frac{|m_1|}{m_1} \Bigl(\sum_j (B A)^j\Bigr) + \frac{|n_1|}{n_1}A^{-1} C \Bigl(\sum_j (B C)^j\Bigr) 
 + \frac{|m_1 n_1|}{m_1 n_1}\Bigl(\sum_j (B A)^j\Bigr) \Bigl(\sum_j (B C)^j\Bigr)BC \biggr| &  $if$ \ \beta_1 > 0,\\
\biggl| \frac{|m_1|}{m_1} \Bigl(\sum_j (B A)^j\Bigr) + \frac{|n_1|}{n_1}  A C^{-1} \Bigl(\sum_j (B C)^j\Bigr) 
 - \frac{|m_1 n_1|}{m_1 n_1}\Bigl(\sum_j (B A)^j\Bigr) \Bigl(\sum_j (B C)^j\Bigr) \biggr| &  $if$ \ \beta_1 < 0.
\end{cases}
\end{align*}
We also have
\begin{align*} 
|M| |N| = \prod^k_{i=2} \left| \sum^{|m_i|}_{j=1} (X_{2i-3}^{(-1)^i} X_{2i-4}^{(-1)^i})^j \right|
\prod^l_{i=2} \left| \sum^{|n_i|}_{j=1} (Y_{2i-3}^{(-1)^i} Y_{2i-4}^{(-1)^i})^j \right| .
\end{align*} 
Hence we obtain the formula of the statement.
Similarly, we can compute $\left|
\begin{array}{cc}
\hspace{-1mm} M_{2k+1}^{\min} & \hspace{-2mm} M_{2k+1}^{' \min} \hspace{-1mm} \\
\hspace{-1mm} N_{2l+1}^{\min} & \hspace{-2mm} N_{2l+1}^{' \min} \hspace{-1mm} 
\end{array}
\right|.$
\end{proof}

Then, we have
\begin{align*}
\displaystyle \Delta_{K,\rho}(t)
& = 
 \frac{\left|
\begin{array}{cc}
M_{2k+1} & M_{2k+1}'\\
N_{2l+1} & N_{2l+1}'
\end{array}
\right|}{t^2 - \mathrm{tr}  \rho(c) t +1}
\\
&\doteq 
  \begin{cases}
  \kappa_0 t^0 + \cdots  + \kappa_0 t^{2(k+l+1)} & $if$ \ m_0 \neq 0,\\
  \lambda_0 t^0 + \cdots + \lambda_0 t^{2(k+l)-2} & $if$ \ m_0 = 0.\\
  \end{cases}
\end{align*}
This completes the proof of Theorem \ref{the case 2}.

\subsection{Examples}
\begin{rem}
If we denote the Alexander polynomial of $K$ by $\Delta_K(t)$, then we have
\begin{eqnarray*}
\Delta_{K}(t)=
  \begin{cases}
 \kappa_0 t^0+ \cdots + \kappa_0 t^{k+l+2} & $if$ \ m_0 \neq 0,\\
 \lambda_0 t^0+ \cdots + \lambda_0 t^{k+l} & $if$ \ m_0 = 0,
   \end{cases}
\end{eqnarray*}
where
\begin{align*}
\kappa_0 &=
|m_k \cdots m_0||n_l \cdots n_1|,\\
\lambda_0 &=
\begin{cases}
|m_k \cdots m_2||n_l \cdots n_2||m_1+ n_1+ m_1 n_1| &  $if$ \ \beta_1 > 0,\\
|m_k \cdots m_2||n_l \cdots n_2||m_1+ n_1- m_1 n_1| &  $if$ \ \beta_1 < 0.
\end{cases}
\end{align*}
Then by \cite{HM}, it is known that the genus of $K$ is given by
\begin{eqnarray*}
2 g(K)=
  \begin{cases}
 k+l+2 & $if$ \ m_0 \neq 0,\\
 k+l & $if$ \ m_0 = 0.
   \end{cases}
\end{eqnarray*}
\end{rem}

It is known that
\begin{align*}
\deg(\Delta_{K}(t)) & \le  2 g(K),\\
\deg(\Delta_{K,\rho}(t)) & \le  4 g(K) -2,
\end{align*}
for any knot $K$.
Furthermore, if $K$ is fibered, then
\begin{align*}
\deg(\Delta_{K}(t)) & =  2 g(K),\\
\deg(\Delta_{K,\rho}(t)) & =  4 g(K) -2,
\end{align*}
 and both $\Delta_{K}(t)$ and $\Delta_{K, \rho}(t)$ are monic.

In the following examples,
we assume that $\beta_1 > 0$, $m_0= 0$ and
\[
|m_2| = \cdots = |m_k| = |n_2| = \cdots =|n_l|=1.
\]
Then their leading coefficients of their Alexander polynomials are 
\begin{align*}
\lambda_0 &=
|m_1+ n_1+ m_1 n_1|
\end{align*}

\begin{exa}
If $(m_1,n_1)=(-2,-2)$, then $\deg(\Delta_{K}(t))$ is less than $2 g(K)$.
On the other hand, the leading coefficient of $\Delta_{K, \rho}(t)$ is
\begin{eqnarray*}
\lambda_0 = |(BA)^{-1} (BC)^{-1} - E|
= 2 - \mathrm{tr}BABC.
\end{eqnarray*}
Thus, if there exists a representation $\rho$ which gives $\mathrm{tr}BABC \neq 2$, then
\[
\deg(\Delta_{K,\rho}(t)) = 4 g(K) -2.
\]
\end{exa}

\begin{exa}
If $(m_1,n_1)=(-2, -3)$,
then $\Delta_K(t)$ is monic but $K$ is not fibered (see \cite{HM}).
On the other hand, the leading coefficient of $\Delta_{K, \rho}(t)$ is
\begin{eqnarray*}
\lambda_0 = |(BA)^{-1} (BC)^{-1} - E|
= 1 - \mathrm{tr}(BABC-E) (BC+E).
\end{eqnarray*}
Thus, if there exists a representation $\rho$ which gives $\mathrm{tr}(BABC-E) (BC+E) \neq 0$, then $\Delta_{K, \rho}(t)$ can't be monic.
\end{exa}

\section{The case of (3) }

In this section, we give the presentation of knot groups of knots which satisfiy the condition (3) by using links whose surgery along the trivial component gives these knots. 
With the presentation, we compute their twisted Alexander polynomials associated to their $SL_2(\bC)$-representations.

In this case, we consider knots $K_n=M(0;(3n+2,-2n-1), (3,1), (3,1))$ depicted in Figure 7.

\begin{figure}[h]
  \begin{center}
\includegraphics[clip,width=4.5cm]{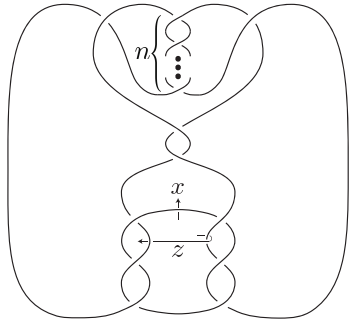}
\caption{The knot $K_n$}
  \end{center}
\end{figure}

For simplicity, we put
\begin{eqnarray*}
X=\rho(x), \  Z=\rho(z), \
W=\begin{cases}
\rho( x^{-1} [x, z][x^{-1},  z^{-1}] x) &n\ $is even$,\\
\rho( [z , x^{-1}][ z^{-1} , x] ) & n\ $is odd$.
\end{cases}
\end{eqnarray*}
If $n$ is even, we set 
\begin{eqnarray*}
A=
\begin{cases}
E & n> 0,\\
-W^{\frac{n}{2}} & n< 0,
\end{cases}
\end{eqnarray*}
and if $n$ is odd, we set
\begin{eqnarray*}
A=
\begin{cases}
E & n> -1,\\
-W^{\frac{n+1}{2}} & n< -1,
\end{cases}
\end{eqnarray*}
where $E$ denotes the identity matrix.

\begin{thm} \label{the case 3}
The twisted Alexander polynomial of $K_n=M(0;(3n+2,-2n-1), (3,1), (3,1))$ associated to their non-abelian
representation $\rho : G(K_n) \to SL_2(\bC)$ is given by
\begin{eqnarray*}
\Delta_{K_n,\rho}(t)
= 
  \begin{cases}
  \kappa_0 t^0 + \cdots   +\kappa_7 t^7  + \cdots + \kappa_0 t^{14} & n\ $is even$,\\
   \lambda_0 t^0 + \cdots +\lambda_3 t^3 +\cdots + \lambda_0 t^6 & n\ $is odd$,
  \end{cases}
\end{eqnarray*}
where $\kappa_i,\lambda_j \in \bC$ are given by the following:
\begin{align*}
\kappa_{0}=&  \Biggl|\sum^{|\frac{n}{2}|}_{i=1} W^i \Biggr|,
\kappa_{1}= -\sum^{|\frac{n}{2}|}_{i=1} tr A W^i X^{-1} [Z^{-1}, X^{-1}]- \Biggl|\sum^{|\frac{n}{2}|}_{i=1} W^i 
\Biggr| tr X,
\kappa_{2}= 1+\sum^{|\frac{n}{2}|}_{i=1} tr A W^i - \Biggl|\sum^{|\frac{n}{2}|}_{i=1} W^i \Biggr|,\\
\kappa_{3}=&  \Biggl|\sum^{|\frac{n}{2}|}_{i=1} W^i \Biggr| \bigl\{ tr XZ + tr W X Z X Z^{-1} W^{-\frac{n}{2}} Z X Z^{-1} W^{\frac{n}{2}} \bigr\},\\
\kappa_{4}=& -\sum^{|\frac{n}{2}|}_{i=1} \bigl\{ tr XZ tr A W^{i-1} Z X^{-1} Z^{-1} + tr A W^i X Z X Z^{-1} \bigr\}  
-\Biggl|\sum^{|\frac{n}{2}|}_{i=1} W^i \Biggr| \bigl\{ tr Z +2tr X^2 Z + tr X Z X Z^{-1}  \bigr\},\\
\kappa_{5}=& tr XZ + \sum^{|\frac{n}{2}|}_{i=1} \bigl\{ tr XZ tr A W^{i} - tr X tr A W^{i-\frac{n}{2}} X^{-1} Z^{-1} X^{-1} + 2 tr A W^{i+1} X^{-1} Z^{-1} X^{-1} W^{-\frac{n}{2}} Z X Z^{-1} \bigr\}\\
&+ \Biggl|\sum^{|\frac{n}{2}|}_{i=1} W^i \Biggr| \bigl\{ tr XZ + tr W^2 X^{-1} Z^{-1} X^{-1} W^{-\frac{n}{2}} Z X Z^{-1} W^{\frac{n}{2}} \bigr\},\\
\kappa_{6}=& tr W^{\frac{n}{2}} XZX+ \sum^{|\frac{n}{2}|}_{i=1} \bigl\{tr A W^{i-\frac{n}{2}} (X^{-1} Z^{-1})^2 -tr A W^{i} (X^{-1} Z^{-1})^2 +tr A W^{i-\frac{n}{2}} X^{-1} Z^{-1} X^{-1}   \bigr\}\\
&+ \Biggl|\sum^{|\frac{n}{2}|}_{i=1} W^i \Biggr| \bigl\{ 2 + tr XZ  tr W X Z X Z^{-1} W^{-\frac{n}{2}} Z X Z^{-1} W^{\frac{n}{2}} \bigr\},\\
\kappa_{7}=&  -\sum^{|\frac{n}{2}|}_{i=1} \bigl\{ tr XZ (tr A W^{i-1} Z + tr A W^{i} X Z X Z^{-1}) + tr A W^{i-1}  Z X^2 Z X -tr A W^{i} X^{-1} Z^2 \bigr\}\\
&- \Biggl|\sum^{|\frac{n}{2}|}_{i=1} W^i \Biggr| \bigl\{ tr X +tr XZ(tr X^2 Z + tr X Z X Z^{-1}) + tr W Z^{-1} W^{-\frac{n}{2}} Z X Z^{-1} W^{\frac{n}{2}} \bigr\},
\end{align*}
and
\begin{align*}
\lambda_{0}=& \Biggl| \sum^{|\frac{n+1}{2}|}_{i=1} W^i \Biggr|,\\
\lambda_{1}=& \sum^{|\frac{n+1}{2}|}_{i=1} tr A W^{i-1} Z X^{-1} Z^{-1} + \Biggl| \sum^{|\frac{n+1}{2}|}_{i=1} W^i \Biggr| \biggl\{ tr X Z^{-1} -tr X + tr W X^2 Z^{-1} W^{\frac{n+1}{2}} Z X Z^{-1} W^{-\frac{n+1}{2}} \biggr\},\\
\lambda_{2}=& 1+ \sum^{|\frac{n+1}{2}|}_{i=1}  \bigl\{ tr A W^{i} X^2 Z^{-1} -tr A W^{i-1} +tr A W^{i-1} Z X^{-2} \bigr\}\\
& + \Biggl| \sum^{|\frac{n+1}{2}|}_{i=1} W^i \Biggr| \bigl\{ 3 -tr X tr W X^2 Z^{-1} W^{\frac{n+1}{2}} Z X Z^{-1} W^{-\frac{n+1}{2}} \\
& \hspace{50mm}+tr X Z^{-1} (tr W X^2 Z^{-1} W^{\frac{n+1}{2}} Z X Z^{-1} W^{-\frac{n+1}{2}} -tr X) \bigr\},\\
\lambda_{3}=& tr X Z^{-1} -trX\\
&+ \sum^{|\frac{n+1}{2}|}_{i=1} \biggl\{ tr A W^{i-1}(Z X Z^{-1} )^{-1} W (X Z X^{-1} )^{-1}  
-tr A W^{i-\frac{n+1}{2}}  (X Z X^{-1}) W^{\frac{n+1}{2}-1} (Z X Z^{-1})^{-1} \\
& \hspace{15mm} -tr A W^i (X Z X^{-1})^{-1}(Z X Z^{-1})
+tr A W^{i-\frac{n+1}{2}} (X Z X^{-1})^{-1} W^{\frac{n+1}{2}} (Z X Z^{-1}) \\
& \hspace{80mm}+ ( tr X Z^{-1} -trX)tr A W^i X^2 Z^{-1} \biggr\}\\
&+ \Biggl| \sum^{|\frac{n+1}{2}|}_{i=1} W^i \Biggr| \biggl\{ 2(tr X Z^{-1} -tr X + tr W X^2 Z^{-1} W^{\frac{n+1}{2}} Z X Z^{-1} W^{-\frac{n+1}{2}} )\\
& \hspace{50mm}- tr X Z^{-1} tr X  tr W X^2 Z^{-1} W^{\frac{n+1}{2}} Z X Z^{-1} W^{-\frac{n+1}{2}}  \biggr\}.
\end{align*}
\end{thm}

\subsection{The case when $n$ is even}
In this case, knots $K_n$ are obtained by $-\frac{1}{n/2}$-surgery along the trivial component of the link $L_e$ depicted in Figure 8.

\begin{figure}[h]
  \begin{center}
\includegraphics[clip,width=4.5cm]{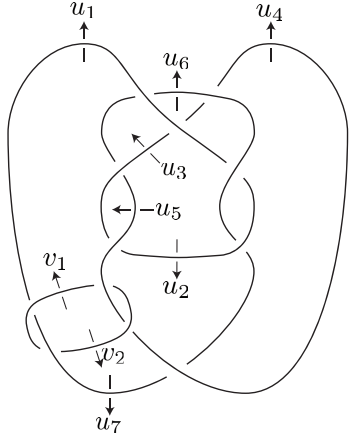}
\caption{The link $L_e$}
  \end{center}
\end{figure}

Then, we obtain 9 relations of the link $L_e$:
\begin{eqnarray}
u_1 u_6 u_2^{-1} u_6^{-1}=1  &\Leftrightarrow& u_2 = u_6^{-1} u_1 u_6\\
u_3 u_5^{-1} u_2^{-1} u_5=1 &&\\
u_3 u_1 u_6 u_4^{-1} u_6^{-1} u_1^{-1}=1 &\Leftrightarrow& u_1 = (u_3 u_1) u_6 u_4^{-1} u_6^{-1}\\
v_2 u_4^{-1} v_2^{-1} u_5=1 &\Leftrightarrow& u_4 = v_2^{-1} u_5 v_2\\
v_1 u_5^{-1} v_2^{-1} u_5=1 &\Leftrightarrow& v_2 = u_5 v_1u_5^{-1} \\
u_3 u_1 u_6 u_1^{-1} u_3^{-1} u_5^{-1} =1 &\Leftrightarrow& u_5 = (u_3 u_1) u_6  (u_3 u_1)^{-1} \\
u_6 u_2 u_4 u_7^{-1} u_4^{-1} u_2^{-1}=1 &\Leftrightarrow& u_7 = (u_2 u_4)^{-1} u_6 u_2 u_4\\
v_1 u_7^{-1} v_2^{-1} u_7=1 &&\\
v_1 u_1 v_1^{-1} u_7^{-1}=1 &&
\end{eqnarray}
From the relations (1), (3), (4), (5), (6) and (7), we can write 
\begin{align*}
u_1 =& (u_3 u_1) u_6 (u_3 u_1) u_6  (u_3 u_1)^{-1} v_1^{-1} (u_3 u_1) u_6^{-1}  (u_3 u_1)^{-1} v_1 (u_3 u_1) u_6^{-1}  (u_3 u_1)^{-1} u_6^{-1},\\
u_2 =& u_6^{-1} (u_3 u_1) u_6 (u_3 u_1) u_6  (u_3 u_1)^{-1} v_1^{-1} (u_3 u_1) u_6^{-1}  (u_3 u_1)^{-1} v_1 (u_3 u_1) u_6^{-1}  (u_3 u_1)^{-1},\\
u_3 =& (u_3 u_1) u_6 (u_3 u_1) u_6  (u_3 u_1)^{-1} v_1^{-1} (u_3 u_1) u_6  (u_3 u_1)^{-1} v_1 (u_3 u_1) u_6^{-1}  (u_3 u_1)^{-1} u_6^{-1} (u_3 u_1)^{-1} ,\\ 
u_4 =& (u_3 u_1) u_6  (u_3 u_1)^{-1} v_1^{-1} (u_3 u_1) u_6  (u_3 u_1)^{-1} v_1 (u_3 u_1) u_6^{-1}  (u_3 u_1)^{-1},\\
u_5 =& (u_3 u_1) u_6  (u_3 u_1)^{-1}, \\
u_7 =& u_6^{-1} (u_3 u_1)^{-1} u_6 (u_3 u_1) u_6, \\
v_2 =& (u_3 u_1) u_6  (u_3 u_1)^{-1} v_1 (u_3 u_1) u_6^{-1}  (u_3 u_1)^{-1},
\end{align*}
with generators $u_3 u_1 , u_6$ and $v_1$.
Since we can obtain relation (8) from relations (2) and (9), we can give a presentation $\pi_1(S^3 - L_e)$ by
\[
\langle u_6, u_3 u_1, v_1  \ | \  r_1, r_2 \rangle.
\]
where
\begin{align*}
r_1=& \bigl( (u_3 u_1) u_6 \bigr)^2 (u_3 u_1)^{-1} v_1^{-1} (u_3 u_1) u_6  (u_3 u_1)^{-1} v_1 (u_3 u_1) \bigl( (u_3 u_1) u_6 \bigr)^{-2}\\
& v_1^{-1} (u_3 u_1) u_6 (u_3 u_1)^{-1} v_1 (u_3 u_1) \bigl( (u_3 u_1) u_6 \bigr)^{-2} u_6 (u_3 u_1) u_6  (u_3 u_1)^{-1},\\
r_2=& v_1  \bigl( (u_3 u_1) u_6 \bigr)^2 (u_3 u_1)^{-1} v_1^{-1} (u_3 u_1) u_6^{-1}  (u_3 u_1)^{-1} v_1 (u_3 u_1) u_6^{-1}  (u_3 u_1)^{-1} u_6^{-1} v_1^{-1} u_6^{-1} (u_3 u_1)^{-1} u_6^{-1} (u_3 u_1) u_6.
\end{align*}
Since we can get $K_n$ from $L_e$ by $-1/\frac{n}{2}$-surgery,
we can obtain the presentation of the knot group $G(K_n)$ by using a presentation of $\pi_1(S^3 - L_e)$, i.e. we have
\begin{align*}
v_1 &= (u_5^{-1} u_7)^{\frac{n}{2}}\\
&= \bigl\{(u_3 u_1) u_6^{-1} (u_3 u_1)^{-1} u_6^{-1} (u_3 u_1)^{-1} u_6 (u_3 u_1) u_6 \bigr\}^{\frac{n}{2}}\\
&= u_6^{-1} \bigl\{ [u_6, (u_3 u_1)] [u_6^{-1}, (u_3 u_1)^{-1}] \bigr\}^{\frac{n}{2}} u_6 .
\end{align*}
Then, we can reduce $r_1$ and get
\begin{align*}
G(K_n)=
\langle x , z  \ | \  [x^{-1}, z^{-1}] x y x (z x z^{-1}) y^{-1}
= y z x (z x z^{-1}) y^{-1} (z x z^{-1})^{-1} \rangle,
\end{align*}
where $x=u_6, z = u_3 u_1$ and $y=x^{-1} ([x, z][x^{-1}, z^{-1}])^{\frac{n}{2}} x$.
By using this presentation, we get
\begin{align*}
\frac{\partial r}{\partial x}  =& \ 
x^{-1} z^{-1} - x^{-1} + [x^{-1},z^{-1}] +  [x^{-1}, z^{-1}] x y  - y z
+ [x^{-1}, z^{-1}] x y x z -y z x z + y z x (z x z^{-1}) y^{-1} z x^{-1}\\
& + \bigl\{-1 + [x^{-1}, z^{-1}] x -  [x^{-1}, z^{-1}] x y x (z x z^{-1}) y^{-1} + y z x (z x z^{-1}) y^{-1} \bigr\} \frac{\partial y}{\partial x} .
\end{align*}

We now compute $(\rho \otimes \frak{a}) \circ \phi \left(\frac{\partial r}{\partial x} \right)$.
Let $\rho$ be a $SL_2(\bC)$-representation of $G(K_n)$.
Then we put 
\begin{align*}
X=\rho(x),\  Y=\rho(y), \  Z=\rho(z), \
W=\rho(x^{-1} [x, z][x^{-1},  z^{-1}] x). 
\end{align*}
Since $\alpha(x)=t , \alpha(z)=t^2$, we can write
\begin{eqnarray*}
\Phi \left(\frac{\partial y}{\partial x} \right) =A B (t^{-4} X^{-1} Z^{-1} X^{-1} - t^{-2} X^{-1} Z^{-1} X^{-1} Z
 + t^{-1} X^{-1} -t X^{-1} [Z^{-1}, X^{-1}]Z ), 
\end{eqnarray*}
where
\[
A=
\begin{cases}
E & n<0\\
-Y & n>0
\end{cases}\ ,\ 
B=\sum_{i=1}^{|n/2|} W^i.
\]

\begin{prop}
We have
\begin{eqnarray*}
\Phi \left(\frac{\partial r}{\partial x} \right) 
 = t^{-4} M + t^{-3} N + t^{-2} O + t^{-1} P + t^0 Q 
 + t^1 R + t^2 S + t^3 T + t^4 U + t^5 V,
\end{eqnarray*}
where
\begin{align*}
M =& A B X^{-1} Z^{-1} X^{-1},\\
N =& X^{-1} Z^{-1} + [X^{-1}, Z^{-1}] X A B X^{-1} Z^{-1} X^{-1}, \\
O =& A B X^{-1} Z^{-1} X^{-1} Z,\\
P =& - X^{-1}- A B X^{-1}-  [X^{-1}, Z^{-1}] X A B X^{-1} Z^{-1} X^{-1} Z
 - [ X^{-1}, Z^{-1}] X Y X (Z X Z^{-1}) Y^{-1} A B X^{-1} Z^{-1} X^{-1},\\
Q =& [ X^{-1}, Z^{-1}] + [ X^{-1}, Z^{-1}] X A B X^{-1}+ Y Z X (Z X Z^{-1}) Y^{-1} A B X^{-1} Z^{-1} X^{-1},\\
R =& [X^{-1}, Z^{-1}] X Y + A B X^{-1} Z^{-1} X^{-1} Z X Z
 + [X^{-1}, Z^{-1}] X Y X (Z X Z^{-1}) Y^{-1} A B X^{-1} Z^{-1} X^{-1} Z,\\
S =& - Y Z  - Y Z X (Z X Z^{-1}) Y^{-1} A B X^{-1} Z^{-1} X^{-1} Z
- [ X^{-1}, Z^{-1}] X A B X^{-1} [ Z^{-1}, X^{-1}] Z \\
&- [ X^{-1}, Z^{-1}] X Y X (Z X Z^{-1}) Y^{-1} A B X^{-1},\\
T =& Y Z X (Z X Z^{-1}) Y^{-1} A B X^{-1},\\
U =& [ X^{-1}, Z^{-1}] X Y X Z  + [ X^{-1}, Z^{-1}] X Y X (Z X Z^{-1}) Y^{-1} A B X^{-1} [ Z^{-1}, X^{-1}] Z,\\
V =& -Y Z X Z + Y Z X (Z X Z^{-1}) Y^{-1} Z X^{-1}
 - Y Z X (Z X Z^{-1}) Y^{-1} A B X^{-1} Z^{-1} X^{-1} Z X Z.\\
\end{align*}
\end{prop}

Then, the following proposition follows from the relation;
\[
|A+B|=|A| + |B| + \mathrm{tr} A B^{*},
\]
where $A, B \in M_2(\bC)$ and $B^{*}$ is the cofactor matrix of $B$.

\begin{prop}
\begin{align*}
\left| \Phi \left(\frac{\partial r}{\partial x} \right) \right| &=\det (t^{-4} M + t^{-3} N + t^{-2} O + t^{-1} P + t^0 Q + t^1 R + t^2 S + t^3 T + t^4 U + t^5 V)\\
&= \sum_{i=-8}^{10} t^i k_i,
\end{align*}
where
\begin{align*}
k_{-8} =& |M|,\\ 
k_{-7} =& \mathrm{tr} M N^{*} ,\\ 
k_{-6} =& |N| + \mathrm{tr} M O^{*}, \\
k_{-5} =& \mathrm{tr} M P^{*} + \mathrm{tr} N O^{*} ,\\ 
k_{-4} =& |O| + \mathrm{tr} M Q^{*} + \mathrm{tr} N P^{*} ,\\
k_{-3} =& \mathrm{tr} M R^{*} + \mathrm{tr} N Q^{*} + \mathrm{tr} O P^{*},\\
k_{-2} =& |P| +\mathrm{tr} M S^{*} + \mathrm{tr} N R^{*} + \mathrm{tr} O Q^{*},\\
k_{-1} =& \mathrm{tr} M T^{*} + \mathrm{tr} N S^{*} + \mathrm{tr} O R^{*} + \mathrm{tr} P Q^{*}, \\
k_{0} =& |Q| + \mathrm{tr} M U^{*} + \mathrm{tr} N T^{*} + \mathrm{tr} O S^{*} + \mathrm{tr} P R^{*},\\
k_{1} =& \mathrm{tr} M V^{*} + \mathrm{tr} N U^{*} + \mathrm{tr} O T^{*} + \mathrm{tr} P S^{*} + \mathrm{tr} Q R^{*},\\
k_{2} =& |R| + \mathrm{tr} N V^{*} + \mathrm{tr} O U^{*} + \mathrm{tr} P T^{*} + \mathrm{tr} Q S^{*},\\
k_{3} =& \mathrm{tr} O V^{*} + \mathrm{tr} P U^{*} + \mathrm{tr} Q T^{*} + \mathrm{tr} R S^{*},\\
k_{4} =& |S| + \mathrm{tr} P V^{*} + \mathrm{tr} Q U^{*} + \mathrm{tr} R T^{*},\\
k_{5} =& \mathrm{tr} Q V^{*} + \mathrm{tr} R U^{*} + \mathrm{tr} S T^{*}, \\
k_{6} =& |T| + \mathrm{tr} R V^{*} + \mathrm{tr} S U^{*},\\
k_{7} =& \mathrm{tr} S V^{*} + \mathrm{tr} T U^{*}, \\ 
k_{8} =& |U| + \mathrm{tr} T V^{*}, \\ 
k_{9} =& \mathrm{tr} U V^{*} ,\\
k_{10} =& |V|.
\end{align*}
\end{prop}

Since we have
\begin{eqnarray*}
\left| \Phi ( z - 1 ) \right| = \left| t^2 Z - E \right| 
= t^0 -  t^2 \mathrm{tr}Z + t^4,
\end{eqnarray*}
we obtain
\begin{align*}
\Delta_{K_n, \rho} (t)&=
 \frac{\left| \displaystyle \Phi \left(\frac{\partial r}{\partial x} \right) \right|}{\left| \Phi ( z - 1 ) \right|} \\
&= \frac{\displaystyle \sum_{i=-8}^{10} t^i k_i}{t^0 -  t^2 \mathrm{tr}Z + t^4}\\
& \sim  \sum_{i=0}^{6} \kappa_i (t^i + t^{14-i})  + \kappa_7 t^7 ,
\end{align*}
where
\begin{align*}
\kappa_0 =& k_{-8} = k_{10},\\
\kappa_1 =& k_{-7} = k_{9}, \\
\kappa_2 =& k_{-6} + \mathrm{tr}Z \kappa_0 = k_{8} + \mathrm{tr}Z \kappa_0,\\
\kappa_3 =& k_{-5} + \mathrm{tr}Z \kappa_1 = k_{7} + \mathrm{tr}Z \kappa_1,\\
\kappa_4 =& k_{-4} - \kappa_0 + \mathrm{tr}Z \kappa_2 = k_{6} - \kappa_0 + \mathrm{tr}Z \kappa_2,\\
\kappa_5 =& k_{-3} - \kappa_1 + \mathrm{tr}Z \kappa_3 = k_{5} - \kappa_1 + \mathrm{tr}Z \kappa_3,\\
\kappa_6 =& k_{-2} - \kappa_2 + \mathrm{tr}Z \kappa_4 = k_{4} - \kappa_2 + \mathrm{tr}Z \kappa_4,\\
\kappa_7 =& k_{-1} - \kappa_3 + \mathrm{tr}Z \kappa_5 = k_{3} - \kappa_3 + \mathrm{tr}Z \kappa_5.
\end{align*}
This proves the even case of Theorem \ref{the case 3}.

\subsection{The case when $n$ is odd}
In this case, knots $K_n$ are obtained by $-\frac{1}{(n+1)/2}$-surgery along the trivial component of the link $L_o$ depicted in Figure 9.

\begin{figure}[h]
  \begin{center}
\includegraphics[clip,width=4.5cm]{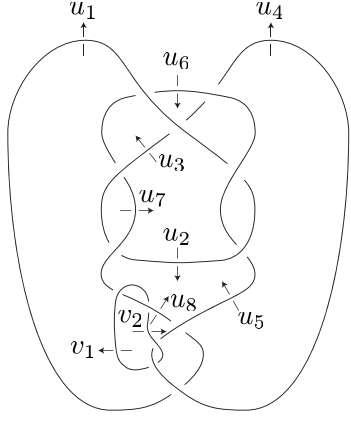}
\caption{The link $L_o$}
  \end{center}
\end{figure}

Then, we obtain 10 relations of the link $L_o$:
\begin{eqnarray}
u_2 u_6 u_1^{-1} u_6^{-1}=1 &\Leftrightarrow& u_2 = u_6 u_1 u_6^{-1} \\
u_6 u_2 u_5^{-1} u_2^{-1} =1 &\Leftrightarrow& u_5 = u_2^{-1} u_6 u_2\\
u_7 u_2 u_7^{-1} u_3^{-1}=1 &&\\
u_3 u_1 u_6^{-1} u_4^{-1} u_6 u_1^{-1}=1 &&\\
v_2 u_4 v_1^{-1} u_4^{-1}=1 &\Leftrightarrow& u_4 = v_2^{-1} u_4 v_1\\
u_4 v_2 u_5^{-1} v_2^{-1} =1 &\Leftrightarrow& u_4 = v_2 u_5 v_2^{-1}\\
u_7 u_3 u_1 u_6^{-1} u_1^{-1} u_3^{-1}=1 &\Leftrightarrow& u_7 = (u_3 u_1) u_6 (u_3 u_1)^{-1} \\
v_1 u_7 v_1^{-1} u_8^{-1}=1  &\Leftrightarrow& u_8 = v_1 u_7 v_1^{-1}\\
v_2 u_8 v_1^{-1} u_8^{-1}=1 &\Leftrightarrow& v_2 = u_8 v_1 u_8^{-1}\\
u_5 u_8 u_5^{-1} u_4 u_1^{-1} u_4^{-1}=1&\Leftrightarrow& u_1 = (u_5^{-1} u_4)^{-1} u_8  (u_5^{-1} u_4)
\end{eqnarray}
From the relations (10), (11), (14), (15), (16), (17), (18) and (19), we can write 
\begin{align*}
u_1 =& (u_3 u_1) u_6 (u_3 u_1)^{-1}[v_1, (u_3 u_1) u_6 (u_3 u_1)^{-1}],\\
u_2 =& u_6 (u_3 u_1) u_6 (u_3 u_1)^{-1}[v_1, (u_3 u_1) u_6 (u_3 u_1)^{-1}] u_6^{-1},\\
u_3 =& (u_3 u_1) [(u_3 u_1) u_6 (u_3 u_1)^{-1}, v_1] (u_3 u_1) u_6^{-1} (u_3 u_1)^{-1},\\
u_4 =& v_1 [(u_3 u_1) u_6 (u_3 u_1)^{-1}, v_1] u_6 [(u_3 u_1) u_6 (u_3 u_1)^{-1}, v_1] [(u_3 u_1), u_6^{-1}] u_6\\
& [u_6^{-1}, (u_3 u_1)] [v_1, (u_3 u_1) u_6 (u_3 u_1)^{-1}] u_6^{-1}  [v_1, (u_3 u_1) u_6 (u_3 u_1)^{-1}] v_1^{-1},\\
u_5 =& u_6 [(u_3 u_1) u_6 (u_3 u_1)^{-1}, v_1] [(u_3 u_1), u_6^{-1}] u_6 [u_6^{-1}, (u_3 u_1)] [v_1, (u_3 u_1) u_6 (u_3 u_1)^{-1}] u_6^{-1},\\
u_7 =& (u_3 u_1) u_6 (u_3 u_1)^{-1}, \\
u_8 =& v_1 (u_3 u_1) u_6 (u_3 u_1)^{-1} v_1^{-1},\\
v_2 =& v_1 [(u_3 u_1) u_6 (u_3 u_1)^{-1}, v_1],
\end{align*}
with generators $u_3 u_1 , u_6$ and $v_1$.
Since we can get relation (13) from relations (12) and (14), we can give a presentation $\pi_1(S^3 - L_o)$ by
\[
\pi_1(S^3 - L_o)=\langle u_6, u_3 u_1,v_1  \ | \  r_1, r_2 \rangle.
\]
where
\begin{align*}
r_1=& (u_3 u_1) u_6 (u_3 u_1)^{-1} u_6 (u_3 u_1) u_6 (u_3 u_1)^{-1} [v_1, (u_3 u_1) u_6 (u_3 u_1)^{-1}] u_6^{-1} [v_1, (u_3 u_1) u_6 (u_3 u_1)^{-1}] (u_3 u_1)^{-1},\\
r_2=& v_1 (u_3 u_1) u_6 (u_3 u_1)^{-1} v_1^2(u_3 u_1) u_6^{-1} (u_3 u_1)^{-1} v_1^{-1} u_6 [(u_3 u_1) u_6 (u_3 u_1)^{-1}, v_1] [(u_3 u_1), u_6^{-1}] u_6 [u_6^{-1}, (u_3 u_1)] \\
&[v_1, (u_3 u_1) u_6 (u_3 u_1)^{-1}] u_6^{-1}  [v_1, (u_3 u_1) u_6 (u_3 u_1)^{-1}] v_1^{-1}  [(u_3 u_1) u_6 (u_3 u_1)^{-1}, v_1] u_6 [(u_3 u_1) u_6 (u_3 u_1)^{-1}, v_1] \\
&[(u_3 u_1), u_6^{-1}] u_6^{-1} [u_6^{-1}, (u_3 u_1)] [v_1, (u_3 u_1) u_6 (u_3 u_1)^{-1}] u_6^{-1}  [v_1, (u_3 u_1) u_6 (u_3 u_1)^{-1}] v_1^{-1}.
\end{align*}

Since we can get $K_n$ from $L_o$ by $1/\frac{n+1}{2}$-surgery,
we can obtain the presentation of knot group $G(K_n)$ by using a presentation of $\pi_1(S^3 - L_o)$, i.e. we have
\begin{eqnarray*}
v_1 = (u_8^{-1} u_4)^{\frac{n+1}{2}}
= \bigl\{[(u_3 u_1), u_6^{-1}][(u_3 u_1)^{-1} , u_6]  \bigr\}^{\frac{n+1}{2}}.
\end{eqnarray*}
Then, we can reduce $r_2$ and get
\begin{eqnarray*}
G(K_n)=
\langle x , z  \ | \   x z x z^{-1} [y , z x z^{-1}] = z x^{-1} [ z x z^{-1}, y] x \rangle,
\end{eqnarray*}
where $x=u_6, z = u_3 u_1$ and $y= \{[z, x^{-1}][z^{-1} , x]  \}^{(n+1)/2}$.

By using this presentation, we get
\begin{align*}
\frac{\partial r}{\partial x}  =&
1+ z x^{-1} - z x^{-1} [z x z^{-1},y] + x z -z x^{-1} z + z x^{-1} (z x z^{-1}) y z x^{-1}
 + x (z x z^{-1}) y z\\
 & -x (z x z^{-1}) y (z x z^{-1}) y^{-1} z x^{-1}
 + \{ z x^{-1} [z x z^{-1}, y] + (x  - z x^{-1}) (z x z^{-1}) -x (z x  z^{-1}) y (z x z^{-1}) y^{-1} \} \frac{\partial y}{\partial x} .
\end{align*}

We now compute $\Phi \left(\frac{\partial r}{\partial x} \right)$.
Suppose $\rho$ be a $SL_2(\bC)$ representation of $G(K_n)$.
Then we put 
\begin{eqnarray*}
X=\rho(x),\  Y=\rho(y), \  Z=\rho(z), \
W=\rho( [z , x^{-1}][z^{-1} , x] ). 
\end{eqnarray*}
Since $\alpha(x)=t , \alpha(z)=t^2$, we can write $\Phi \left(\frac{\partial y}{\partial x} \right) $ by
\begin{eqnarray*}
A B (t^{-2} X Z^{-1} X^{-1} + t^{-1} W^{-1} (Z X Z^{-1})^{-1} - E -t W^{-1} Z X^{-1} ), 
\end{eqnarray*}
where
\[
A=
\begin{cases}
E & n> -1\\
-Y & n< -1
\end{cases}\ ,\ 
B=\sum_{i=1}^{|(n+1)/2|} W^i.
\]

\begin{prop}
We have
\begin{eqnarray*}
\Phi \left(\frac{\partial r}{\partial x} \right)  = t^{-1} P + t^0 Q + t^1 R + t^2 S + t^3 T + t^4 U,
\end{eqnarray*}
where
\begin{align*}
P =&\ (X Z X^{-1}) W^{-1} Y (Z X Z^{-1})^{-1} Y^{-1} A B (X Z X^{-1})^{-1},\\
Q =&\ E + (X Z X^{-1}) W^{-1} Y (Z X Z^{-1})^{-1} Y^{-1} A B W^{-1} (Z X Z^{-1})^{-1} + X (Z X Z^{-1}) A B  (X Z X^{-1})^{-1}\\
& - (X Z X^{-1}) W^{-1} A B (X Z X^{-1})^{-1},\\
R =&\ Z X^{-1} -Z X^{-1} [Z X Z^{-1}, Y] - (X Z X^{-1}) W^{-1} Y (Z X Z^{-1})^{-1} Y^{-1} A B \\
&+ X (Z X Z^{-1}) A B W^{-1} (Z X Z^{-1})^{-1} - (X Z X^{-1}) W^{-1} A B W^{-1} (Z X Z^{-1})^{-1}\\
& - X (Z X Z^{-1}) Y (Z X Z^{-1}) Y^{-1} A B (X Z X^{-1})^{-1},\\
S =&\ -(X Z X^{-1}) W^{-1} Y (Z X Z^{-1})^{-1} Y^{-1} A B  W^{-1} Z X^{-1} - X (Z X Z^{-1}) A B\\
& + (X Z X^{-1}) W^{-1} A B - X (Z X Z^{-1}) Y (Z X Z^{-1}) Y^{-1} A B W^{-1} (Z X Z^{-1})^{-1},\\
T =&\ X Z -Z X^{-1} Z + (X Z X^{-1}) W^{-1} Y Z X^{-1} - X (Z X Z^{-1}) A B W^{-1} Z X^{-1}\\
& + (X Z X^{-1}) W^{-1} A B W^{-1} Z X^{-1} + X (Z X Z^{-1}) Y (Z X Z^{-1}) Y^{-1} A B,\\
U =&\ X (Z X Z^{-1}) Y Z - X (Z X Z^{-1}) Y (Z X Z^{-1}) Y^{-1} Z X^{-1} \\
&+ X (Z X Z^{-1}) Y (Z X Z^{-1}) Y^{-1} A B W^{-1} Z X^{-1} .
\end{align*}
\end{prop}

\vspace{11mm}

We have
\begin{align*}
\left| \Phi \left(\frac{\partial r}{\partial x} \right) \right| 
=\det (t^{-1} P + t^0 Q + t^1 R + t^2 S + t^3 T + t^4 U)
= \sum_{i=-2}^{8} t^i l_i,
\end{align*}
where
\begin{align*}
l_{-2} &= |P|, \\
l_{-1} &= \mathrm{tr} P Q^{*}, \\
l_{0} &= |Q| + \mathrm{tr} P R^{*}, \\
l_{1} &= \mathrm{tr} P S^{*} + \mathrm{tr} Q R^{*},  \\
l_{2} &= |R| + \mathrm{tr} P T^{*} + \mathrm{tr} Q S^{*}, \\
l_{3} &= \mathrm{tr} P U^{*} + \mathrm{tr} Q T^{*} +\mathrm{tr} R S^{*}, \\
l_{4} &= |S| + \mathrm{tr} Q U^{*} + \mathrm{tr} R T^{*}, \\
l_{5} &= \mathrm{tr} R U^{*} + \mathrm{tr} S T^{*}, \\
l_{6} &= |T|+ \mathrm{tr} S U^{*}, \\
l_{7} &= \mathrm{tr} T U^{*}, \\
l_{8} &= |U|.
\end{align*}

Since we have
\begin{eqnarray*}
\left| \Phi ( z - 1 ) \right| 
= t^0 -  t^2 \mathrm{tr}Z + t^4,
\end{eqnarray*}
we obtain
\begin{align*}
\Delta_{K_n, \rho} (t)
&= \frac{\left| \displaystyle \Phi \left(\frac{\partial r}{\partial x} \right) \right|}{\left| \Phi ( z - 1 ) \right|} \\
&= \frac{\displaystyle \sum_{i=-2}^{8} t^i l_i}{t^0 -  t^2 \mathrm{tr}Z + t^4}\\
& \sim  \sum_{i=0}^{2} \lambda_i (t^i + t^{6-i})  + \lambda_3 t^3 ,
\end{align*}
where
\begin{align*}
\lambda_{0} &= l_{-2} = l_{8}, \\
\lambda_{1} &= l_{-1} = l_{7}, \\
\lambda_{2} &= l_{0} + \mathrm{tr}Z \lambda_0 = l_{6} + \mathrm{tr}Z \lambda_0, \\
\lambda_{3} &= l_{1} + \mathrm{tr}Z \lambda_1 = l_{5} + \mathrm{tr}Z \lambda_1.
\end{align*}

\appendix
\section{Calculations of $R_{i}$ and $R'_{i}$ defined in Section 2}

In Section 2, we defined $R_{-1}, R_{0}, \ldots , R_{2k}$ and $R'_{1}, R'_{2}, \ldots , R'_{2k}$.
In this appendix, we give these calculation.

\begin{align*}
R_{-1} & = \Phi \left( \frac{\partial r_{-1}}{\partial a} \right)=
\begin{cases}
\displaystyle t^{-1} \sum_{j=0}^{m_0-1} X_{-1}(A B)^j - t^{0} \Bigl(\sum_{j=0}^{m_0-1} (A B)^j + (AB)^{m_0} \Bigr)& $if$ \  m_0>0\\
\displaystyle -t^{-1} \sum_{j=m_0}^{-1} X_{-1}(A B)^j + t^{0} \Bigl(\sum_{j=m_0}^{-1} (A B)^j - (AB)^{m_0} \Bigr) & $if$ \  m_0<0\\
\end{cases},\\
R_{0} &= \Phi \left( \frac{\partial r_{0}}{\partial a} \right) =
\begin{cases}
\displaystyle -t^0 \sum_{j=0}^{m_0-1} (A B)^j + t^1 \sum_{j=0}^{m_0-1} X_0 (A B)^j & $if$ \  m_0>0\\
\displaystyle t^0 \sum_{j=m_0}^{-1} (A B)^j - t^1 \sum_{j=m_0}^{-1} X_0 (A B)^j& $if$ \  m_0<0\\
\end{cases} ,\\
R_{1} & = \Phi \left( \frac{\partial r_{1}}{\partial a} \right) \\&=
\begin{cases}
\displaystyle t^0 X_1\sum_{j=1}^{m_1} (X_{-1}^{-1} A)^j A^{-1} - t^1 \sum_{j=1}^{m_1} (X_{-1}^{-1} A)^j A^{-1} & $if$ \  m_1>0\\
\displaystyle -t^0 X_1\sum_{j=m_1+1}^{0} (X_{-1}^{-1} A)^j A^{-1} + t^1 \sum_{j=m_1+1}^{0} (X_{-1}^{-1} A)^j A^{-1} & $if$ \  m_1<0\\
\end{cases} ,\\
R_{2} & = \Phi \left( \frac{\partial r_{2}}{\partial a} \right)\\ & =
\begin{cases}
\displaystyle -t^1\Bigl(\sum_{j=1}^{m_1} (X_{-1}^{-1} A)^j - (X_{-1}^{-1} A)^{m_1} \Bigr)A^{-1} + t^2 X_2 \sum_{j=1}^{m_1} (X_{-1}^{-1} A)^j A^{-1} & $if$ \  m_1>0\\
\displaystyle t^1 \Bigl(\sum_{j=m_1+1}^{0} (X_{-1}^{-1} A)^j + (X_{-1}^{-1} A)^{m_1} \Bigr)A^{-1} - t^2 X_2 \sum_{j=m_1+1}^{0} (X_{-1}^{-1} A)^j A^{-1} & $if$ \  m_1<0\\
\end{cases}, \\
-R'_{1} & =  -\Phi \left( \frac{\partial r_{1}}{\partial x_{-1}} \right) \\&=
\begin{cases}
\displaystyle t^0 X_1\Bigl(\sum_{j=1}^{m_1} (X_{-1}^{-1} A)^j + (X_{-1}^{-1} A)^{m_1+1} \Bigr)A^{-1} - t^1 \sum_{j=1}^{m_1} (X_{-1}^{-1} A)^j A^{-1} & $if$ \  m_1>0\\
\displaystyle -t^0 X_1 \Bigl(\sum_{j=m_1+1}^{0} (X_{-1}^{-1} A)^j - (X_{-1}^{-1} A)^{m_1+1} \Bigr)A^{-1} + t^1 \sum_{j=m_1+1}^{0} (X_{-1}^{-1} A)^j A^{-1} & $if$ \  m_1<0\\
\end{cases} ,\\
-R'_{2} & =  -\Phi \left( \frac{\partial r_{2}}{\partial x_{-1}} \right) \\ &=
\begin{cases}
\displaystyle -t^1 \sum_{j=1}^{m_1} (X_{-1}^{-1} A)^j A^{-1} + t^2 X_2 \sum_{j=1}^{m_1} (X_{-1}^{-1} A)^j A^{-1} & $if$ \  m_1>0\\
\displaystyle t^1 \sum_{j=m_1+1}^{0} (X_{-1}^{-1} A)^j A^{-1} - t^2 X_2 \sum_{j=m_1+1}^{0} (X_{-1}^{-1} A)^j A^{-1} & $if$ \  m_1<0\\
\end{cases} ,
\end{align*}
and if $i$ is even, 
\begin{align*}
&-R_{2i-1}\\=&\ \Phi \left( \frac{\partial r_{2i-1}}{\partial x_{2i-4}} \right) \\ =&
\begin{cases}
\displaystyle -t^{-2} X_{2i-1} \sum_{j=1}^{m_i} (X_{2i-3} X_{2i-4})^j X_{2i-4}^{-1} + t^{-1} \sum_{j=1}^{m_i} (X_{2i-3} X_{2i-4})^j X_{2i-4}^{-1} & $if$ \  m_i>0\\
\displaystyle t^{-2} X_{2i-1} \sum_{j=m_i + 1}^{0} (X_{2i-3} X_{2i-4})^j X_{2i-4}^{-1} - t^{-1} \sum_{j=m_i+1}^{0} (X_{2i-3} X_{2i-4})^j X_{2i-4}^{-1} & $if$ \  m_i<0\\
\end{cases} ,
\end{align*}
\begin{align*}
&-R'_{2i-1}\\=& \Phi \left( \frac{\partial r_{2i-1}}{\partial x_{2i-3}} \right)\\ =&
\begin{cases}
\displaystyle -t^{-1} X_{2i-1} X_{2i-3}^{-1} \sum_{j=1}^{m_i} (X_{2i-3} X_{2i-4})^j X_{2i-4}^{-1} 
+ t^{0} X_{2i-3}^{-1}\Bigl(\sum_{j=1}^{m_i} (X_{2i-3} X_{2i-4})^j + (X_{2i-3} X_{2i-4})^{m_i+1}\Bigr) X_{2i-4}^{-1}\\
& \hspace{-20mm} $if$ \  m_i>0\\
\displaystyle t^{-1} X_{2i-1} X_{2i-3}^{-1} \sum_{j=m_i +1}^{0} (X_{2i-3} X_{2i-4})^j X_{2i-4}^{-1} - t^{0} X_{2i-3}^{-1}\Bigl(\sum_{j=m_i+1}^{0} (X_{2i-3} X_{2i-4})^j -(X_{2i-3} X_{2i-4})^{m_i+1} \Bigr) X_{2i-4}^{-1}\\
& \hspace{-20mm} $if$ \  m_i<0\\
\end{cases} ,\\
&-R_{2i}\\=&\ \Phi \left( \frac{\partial r_{2i}}{\partial x_{2i-4}} \right), \\ =&
\begin{cases}
\displaystyle t^{-1}  \sum_{j=1}^{m_i } (X_{2i-3} X_{2i-4})^j X_{2i-4}^{-1} - t^{0} X_{2i}\Bigl(\sum_{j=1}^{m_i} (X_{2i-3} X_{2i-4})^j - (X_{2i-3} X_{2i-4})^{m_i}\Bigr)X_{2i-4}^{-1} & $if$ \  m_i>0\\
\displaystyle -t^{-1} \sum_{j=m_i+1 }^{0} (X_{2i-3} X_{2i-4})^j X_{2i-4}^{-1} + t^{0} X_{2i}\Bigl(\sum_{j=m_i+1}^{0} (X_{2i-3} X_{2i-4})^j + (X_{2i-3} X_{2i-4})^{m_i} \Bigr)X_{2i-4}^{-1} & $if$ \  m_i<0\\
\end{cases} ,\\
&-R'_{2i}\\=& \Phi \left( \frac{\partial r_{2i}}{\partial x_{2i-3}} \right)\\ =&
\begin{cases}
\displaystyle t^{0} X_{2i-3}^{-1} \sum_{j=1}^{m_i} (X_{2i-3} X_{2i-4})^j X_{2i-4}^{-1} - t^{1} X_{2i} X_{2i-3}^{-1} \sum_{j=1}^{m_i} (X_{2i-3} X_{2i-4})^j X_{2i-4}^{-1} & $if$ \  m_i>0\\
\displaystyle -t^{0} X_{2i-3}^{-1} \sum_{j=m_i +1}^{0} (X_{2i-3} X_{2i-4})^j X_{2i-4}^{-1} + t^{1} X_{2i} X_{2i-3}^{-1} \sum_{j=m_i+1}^{0} (X_{2i-3} X_{2i-4})^j X_{2i-4}^{-1} & $if$ \  m_i<0\\
\end{cases} ,
\end{align*}
for $2 \leq i \leq k$.

If $i$ is odd, 
\begin{align*}
&-R_{2i-1}\\=&\ \Phi \left( \frac{\partial r_{2i-1}}{\partial x_{2i-4}} \right) \\ =&
\begin{cases}
\displaystyle t^{-1} X_{2i-1} \sum_{j=1}^{m_i} (X_{2i-3}^{-1} X_{2i-4}^{-1})^j - t^{0} \sum_{j=1}^{m_i} (X_{2i-3}^{-1} X_{2i-4}^{-1})^j & $if$ \  m_i>0\\
\displaystyle -t^{-1} X_{2i-1} \sum_{j=m_i + 1}^{0} (X_{2i-3}^{-1} X_{2i-4}^{-1})^j + t^{0} \sum_{j=m_i+1}^{0} (X_{2i-3}^{-1} X_{2i-4}^{-1})^j  & $if$ \  m_i<0\\
\end{cases} ,\\
&-R'_{2i-1}\\=& \Phi \left( \frac{\partial r_{2i-1}}{\partial x_{2i-3}} \right)\\ =&
\begin{cases}
\displaystyle t^{0} X_{2i-1}  \Bigl(\sum_{j=1}^{m_i} (X_{2i-3}^{-1} X_{2i-4}^{-1})^j + (X_{2i-3}^{-1} X_{2i-4}^{-1})^{m_i+1}\Bigr) X_{2i-4} - t^{1} \sum_{j=1}^{m_i} (X_{2i-3}^{-1} X_{2i-4}^{-1})^j X_{2i-4} & $if$ \  m_i>0\\
\displaystyle -t^{0} X_{2i-1} \Bigl(\sum_{j=m_i +1}^{0} (X_{2i-3}^{-1} X_{2i-4}^{-1})^j -(X_{2i-3}^{-1} X_{2i-4}^{-1})^{m_i+1} \Bigr) X_{2i-4} + t^{1}\sum_{j=m_i + 1}^{0} (X_{2i-3}^{-1} X_{2i-4}^{-1})^j X_{2i-4} & $if$ \  m_i<0\\
\end{cases} ,\\
&-R_{2i}\\=&\ \Phi \left( \frac{\partial r_{2i}}{\partial x_{2i-4}} \right), \\ =&
\begin{cases}
\displaystyle -t^{0}  \Bigl(\sum_{j=1}^{m_i } (X_{2i-3}^{-1} X_{2i-4}^{-1})^j - (X_{2i-3}^{-1} X_{2i-4}^{-1})^{m_i}\Bigr) + t^{1} X_{2i}\sum_{j=1}^{m_i} (X_{2i-3}^{-1} X_{2i-4}^{-1})^j  & $if$ \  m_i>0\\
\displaystyle t^{0} \Bigl(\sum_{j=m_i+1 }^{0} (X_{2i-3}^{-1} X_{2i-4}^{-1})^j + (X_{2i-3}^{-1} X_{2i-4}^{-1})^{m_i} \Bigr) - t^{1} X_{2i} \sum_{j=m_i+1}^{0} (X_{2i-3}^{-1} X_{2i-4}^{-1})^j  & $if$ \  m_i<0\\
\end{cases} ,
\end{align*}
\begin{align*}
&-R'_{2i}\\=& \Phi \left( \frac{\partial r_{2i}}{\partial x_{2i-3}} \right)\\ =&
\begin{cases}
\displaystyle -t^{1}  \sum_{j=1}^{m_i} (X_{2i-3}^{-1} X_{2i-4}^{-1})^j X_{2i-4} + t^{2} X_{2i} \sum_{j=1}^{m_i} (X_{2i-3}^{-1} X_{2i-4}^{-1})^j X_{2i-4} & $if$ \  m_i>0\\
\displaystyle t^{1}  \sum_{j=m_i +1}^{0} (X_{2i-3}^{-1} X_{2i-4}^{-1})^j X_{2i-4} - t^{2} X_{2i} \sum_{j=m_i+1}^{0} (X_{2i-3}^{-1} X_{2i-4}^{-1})^j X_{2i-4} & $if$ \  m_i<0\\
\end{cases} ,
\end{align*}
for $3 \leq i \leq k$.

\section{Calculations of $R_{i}, R'_{i}$ and $S_{i}, S'_{i}$ defined in section 3}
In Section 3, we defined $R_{-4}$ which corresponds to first rational tangle $\beta_1/\alpha_1$, $R_{-1}, R_{0}, \ldots , R_{2k}$ and $R'_{-1}, R'_{0}, \ldots , R'_{2k}$ which corresponds to scond rational tangle $\beta_2/\alpha_2$, and $ S_{1}, S_{2}, \ldots , S_{2l}$ and $S'_{1}, R'_{2}, \ldots , S'_{2l}$ which corresponds to third rational tangle $\beta_3/\alpha_3$.
In this appendix, we give these calculation.

\begin{align*}
R_{-4} & = \Phi \left( \frac{\partial r_{-4}}{\partial a} \right)=
\begin{cases}
\displaystyle t^{-2} X_{-4} C^{-1}& $if$ \  \beta_1=1\\
\displaystyle -t^{-2} X_{-4} A^{-1} + t^{-1} A^{-1} (E + CA^{-1}) & $if$ \  \beta_1=-1\\
\end{cases}
\end{align*}
and if $i$ is even, 
\begin{align*}
-R_{2i-1}=&\ \Phi \left( \frac{\partial r_{2i-1}}{\partial x_{2i-4}} \right) \\ =&
\begin{cases}
\displaystyle t^{1} \sum_{j=0}^{m_i-1} (X_{2i-3} X_{2i-4})^j X_{2i-3} - t^{2} X_{2i-1} \sum_{j=0}^{m_i-1} (X_{2i-3} X_{2i-4})^j X_{2i-3} & $if$ \  m_i>0\\
\displaystyle -t^{1} \sum_{j=m_i}^{-1} (X_{2i-3} X_{2i-4})^j X_{2i-3} + t^{2} X_{2i-1} \sum_{j=m_i}^{-1} (X_{2i-3} X_{2i-4})^j X_{2i-3} & $if$ \  m_i<0\\
\end{cases} ,\\
-R'_{2i-1}=& \Phi \left( \frac{\partial r_{2i-1}}{\partial x_{2i-3}} \right)\\ =&
\begin{cases}
\displaystyle t^{0} \Bigl(\sum_{j=0}^{m_i-1} (X_{2i-3} X_{2i-4})^j + (X_{2i-3} X_{2i-4})^{m_i}\Bigr)  - t^{1} X_{2i-1} \sum_{j=0}^{m_i-1} (X_{2i-3} X_{2i-4})^j & $if$ \  m_i>0\\
\displaystyle -t^{0} \Bigl(\sum_{j=m_i}^{-1} (X_{2i-3} X_{2i-4})^j - (X_{2i-3} X_{2i-4})^{m_i}\Bigr) + t^{1} X_{2i-1} \sum_{j=m_i}^{-1} (X_{2i-3} X_{2i-4})^j & $if$ \  m_i<0\\
\end{cases} ,\\
-R_{2i}=&\ \Phi \left( \frac{\partial r_{2i}}{\partial x_{2i-4}} \right), \\ =&
\begin{cases}
\displaystyle -t^{0} \Bigl(X_{2i} \sum_{j=0}^{m_i-1} (X_{2i-3} X_{2i-4})^j X_{2i-3} - (X_{2i-3} X_{2i-4})^{m_i}\Bigr)  + t^{1} \sum_{j=0}^{m_i-1} (X_{2i-3} X_{2i-4})^j X_{2i-3} & $if$ \  m_i>0\\
\displaystyle t^{0} \Bigl(X_{2i} \sum_{j=m_i}^{-1} (X_{2i-3} X_{2i-4})^j X_{2i-3} + (X_{2i-3} X_{2i-4})^{m_i}\Bigr) - t^{1} \sum_{j=m_i}^{-1} (X_{2i-3} X_{2i-4})^j X_{2i-3} & $if$ \  m_i<0\\
\end{cases} ,\\
-R'_{2i}=& \Phi \left( \frac{\partial r_{2i}}{\partial x_{2i-3}} \right)\\ =&
\begin{cases}
\displaystyle -t^{-1} X_{2i} \sum_{j=0}^{m_i-1} (X_{2i-3} X_{2i-4})^j + t^{0} \sum_{j=0}^{m_i-1} (X_{2i-3} X_{2i-4})^j & $if$ \  m_i>0\\
\displaystyle t^{-1} X_{2i} \sum_{j=m_i}^{-1} (X_{2i-3} X_{2i-4})^j - t^{0} \sum_{j=m_i}^{-1} (X_{2i-3} X_{2i-4})^j & $if$ \  m_i<0\\
\end{cases} ,
\end{align*}
for $0 \leq i \leq k$.

If $i$ is odd, 
\begin{align*}
-R_{2i-1}=&\ \Phi \left( \frac{\partial r_{2i-1}}{\partial x_{2i-4}} \right) \\ =&
\begin{cases}
\displaystyle -t^{0}  \sum_{j=1}^{m_i} (X_{2i-3}^{-1} X_{2i-4}^{-1})^j + t^{1} X_{2i-1} \sum_{j=1}^{m_i} (X_{2i-3}^{-1} X_{2i-4}^{-1})^j & $if$ \  m_i>0\\
\displaystyle t^{0} \sum_{j=m_i + 1}^{0} (X_{2i-3}^{-1} X_{2i-4}^{-1})^j - t^{1} X_{2i-1} \sum_{j=m_i+1}^{0} (X_{2i-3}^{-1} X_{2i-4}^{-1})^j  & $if$ \  m_i<0\\
\end{cases} ,\\
-R'_{2i-1}=& \Phi \left( \frac{\partial r_{2i-1}}{\partial x_{2i-3}} \right)\\ =&
\begin{cases}
\displaystyle - t^{-1} \sum_{j=1}^{m_i} (X_{2i-3}^{-1} X_{2i-4}^{-1})^j X_{2i-4} + t^{0} X_{2i-1}  \Bigl(\sum_{j=1}^{m_i} (X_{2i-3}^{-1} X_{2i-4}^{-1})^j + (X_{2i-3}^{-1} X_{2i-4}^{-1})^{m_i+1}\Bigr) X_{2i-4}  & $if$ \  m_i>0\\
\displaystyle t^{-1}\sum_{j=m_i + 1}^{0} (X_{2i-3}^{-1} X_{2i-4}^{-1})^j X_{2i-4} -t^{0} X_{2i-1} \Bigl(\sum_{j=m_i +1}^{0} (X_{2i-3}^{-1} X_{2i-4}^{-1})^j -(X_{2i-3}^{-1} X_{2i-4}^{-1})^{m_i+1} \Bigr) X_{2i-4}  & $if$ \  m_i<0\\
\end{cases} ,\\
-R_{2i}=&\ \Phi \left( \frac{\partial r_{2i}}{\partial x_{2i-4}} \right), \\ =&
\begin{cases}
\displaystyle  t^{-1} X_{2i}\sum_{j=1}^{m_i} (X_{2i-3}^{-1} X_{2i-4}^{-1})^j  -t^{0}  \Bigl(\sum_{j=1}^{m_i } (X_{2i-3}^{-1} X_{2i-4}^{-1})^j - (X_{2i-3}^{-1} X_{2i-4}^{-1})^{m_i}\Bigr) & $if$ \  m_i>0\\
\displaystyle  - t^{-1} X_{2i} \sum_{j=m_i+1}^{0} (X_{2i-3}^{-1} X_{2i-4}^{-1})^j  + t^{0} \Bigl(\sum_{j=m_i+1 }^{0} (X_{2i-3}^{-1} X_{2i-4}^{-1})^j + (X_{2i-3}^{-1} X_{2i-4}^{-1})^{m_i} \Bigr) & $if$ \  m_i<0\\
\end{cases} ,\\
-R'_{2i}=& \Phi \left( \frac{\partial r_{2i}}{\partial x_{2i-3}} \right)\\ =&
\begin{cases}
\displaystyle  t^{-2} X_{2i} \sum_{j=1}^{m_i} (X_{2i-3}^{-1} X_{2i-4}^{-1})^j X_{2i-4} - t^{-1}  \sum_{j=1}^{m_i} (X_{2i-3}^{-1} X_{2i-4}^{-1})^j X_{2i-4}  & $if$ \  m_i>0\\
\displaystyle - t^{-2} X_{2i} \sum_{j=m_i+1}^{0} (X_{2i-3}^{-1} X_{2i-4}^{-1})^j X_{2i-4} + t^{-1}  \sum_{j=m_i +1}^{0} (X_{2i-3}^{-1} X_{2i-4}^{-1})^j X_{2i-4} & $if$ \  m_i<0\\
\end{cases} ,
\end{align*}
for $0 \leq i \leq k$.

If $i$ is even,
\begin{align*}
-S_{2i-1}=&\ \Phi \left( \frac{\partial s_{2i-1}}{\partial y_{2i-4}} \right) \\ =&
\begin{cases}
\displaystyle -t^{1} \sum_{j=-n_i}^{-1} (Y_{2i-3} Y_{2i-4})^j Y_{2i-3} + t^{2} Y_{2i-1} \sum_{j=-n_i}^{-1} (Y_{2i-3} Y_{2i-4})^j Y_{2i-3} & $if$ \  n_i>0\\
\displaystyle t^{1} \sum_{j=0}^{-n_i-1} (Y_{2i-3} Y_{2i-4})^j Y_{2i-3} - t^{2} Y_{2i-1}  \sum_{j=0}^{-n_i-1} (Y_{2i-3} Y_{2i-4})^j Y_{2i-3}  & $if$ \  n_i<0\\
\end{cases} ,\\
-S'_{2i-1}=& \Phi \left( \frac{\partial s_{2i-1}}{\partial y_{2i-3}} \right)\\ =&
\begin{cases}
\displaystyle -t^{0} \Bigl(\sum_{j=-n_i}^{-1} (Y_{2i-3} Y_{2i-4})^j - (Y_{2i-3} Y_{2i-4})^{-n_i}\Bigr) +t^{1} Y_{2i-1}  \sum_{j=-n_i}^{-1} (Y_{2i-3} Y_{2i-4})^j  & $if$ \  n_i>0\\
\displaystyle t^{0}  \Bigl(\sum_{j=0}^{-n_i-1} (Y_{2i-3} Y_{2i-4})^j +(Y_{2i-3} Y_{2i-4})^{-n_i} \Bigr) - t^{1} Y_{2i-1} \sum_{j=0}^{-n_i-1} (Y_{2i-3} Y_{2i-4})^j  & $if$ \  n_i<0\\
\end{cases} ,\\
-S_{2i}=&\ \Phi \left( \frac{\partial s_{2i}}{\partial y_{2i-4}} \right), \\ =&
\begin{cases}
\displaystyle t^{0} \Bigl(Y_{2i} \sum_{j=-n_i}^{-1} (Y_{2i-3} Y_{2i-4})^j Y_{2i-3} - (Y_{2i-3} Y_{2i-4})^{-n_i}\Bigr) -t^{1} \sum_{j=-n_i}^{-1} (Y_{2i-3} Y_{2i-4})^j  Y_{2i-3} & $if$ \  n_i>0\\
\displaystyle -t^{0}  \Bigl(Y_{2i}\sum_{j=0}^{-n_i-1} (Y_{2i-3} Y_{2i-4})^j Y_{2i-3}-(Y_{2i-3} Y_{2i-4})^{-n_i} \Bigr) + t^{1} \sum_{j=0}^{-n_i-1} (Y_{2i-3} Y_{2i-4})^j Y_{2i-3} & $if$ \  n_i<0\\
\end{cases} ,
\end{align*}
\begin{align*}
-S'_{2i}=& \Phi \left( \frac{\partial s_{2i}}{\partial y_{2i-3}} \right)\\ =&
\begin{cases}
\displaystyle t^{-1} Y_{2i} \sum_{j=-n_i}^{-1} (Y_{2i-3} Y_{2i-4})^j - t^{0}  \sum_{j=-n_i}^{-1} (Y_{2i-3} Y_{2i-4})^j & $if$ \  n_i>0\\
\displaystyle -t^{-1} Y_{2i} \sum_{j=0}^{-n_i-1} (Y_{2i-3} Y_{2i-4})^j + t^{0} \sum_{j=0}^{-n_i-1} (Y_{2i-3} Y_{2i-4})^j  & $if$ \  n_i<0\\
\end{cases} ,
\end{align*}
for $1 \leq i \leq l$.

If $i$ is odd,
\begin{align*}
-S_{2i-1}=&\ \Phi \left( \frac{\partial s_{2i-1}}{\partial y_{2i-4}} \right) \\ =&
\begin{cases}
\displaystyle t^{0} \sum_{j=-n_i+1}^{0} (Y_{2i-3}^{-1} Y_{2i-4}^{-1})^j  - t^{1} Y_{2i-1} \sum_{j=-n_i+1}^{0} (Y_{2i-3}^{-1} Y_{2i-4}^{-1})^j & $if$ \  n_i>0\\
\displaystyle -t^{0} \sum_{j=1}^{-n_i} (Y_{2i-3}^{-1} Y_{2i-4}^{-1})^j  + t^{1} Y_{2i-1}  \sum_{j=1}^{-n_i} (Y_{2i-3}^{-1} Y_{2i-4}^{-1})^j  & $if$ \  n_i<0\\
\end{cases} ,\\
-S'_{2i-1}=& \Phi \left( \frac{\partial s_{2i-1}}{\partial y_{2i-3}} \right)\\ =&
\begin{cases}
\displaystyle t^{-1} \sum_{j=-n_i+1}^{0} (Y_{2i-3}^{-1} Y_{2i-4}^{-1})^j Y_{2i-4} - t^{0} \Bigl( Y_{2i-1} \sum_{j=-n_i+1}^{0} (Y_{2i-3}^{-1} Y_{2i-4}^{-1})^j Y_{2i-4}- (Y_{2i-3}^{-1} Y_{2i-4}^{-1})^{-n_i}\Bigr) & $if$ \  n_i>0\\
\displaystyle -t^{-1}  \sum_{j=1}^{-n_i} (Y_{2i-3}^{-1} Y_{2i-4}^{-1})^j Y_{2i-4} + t^{0}  \Bigl( Y_{2i-1} \sum_{j=1}^{-n_i} (Y_{2i-3}^{-1} Y_{2i-4}^{-1})^j Y_{2i-4} + (Y_{2i-3}^{-1} Y_{2i-4}^{-1})^{-n_i}\Bigr)  & $if$ \  n_i<0\\
\end{cases} ,\\
-S_{2i}=&\ \Phi \left( \frac{\partial s_{2i}}{\partial y_{2i-4}} \right), \\ =&
\begin{cases}
\displaystyle -t^{-1} Y_{2i} \sum_{j=-n_i+1}^{0} (Y_{2i-3}^{-1} Y_{2i-4}^{-1})^j +t^{0} \Bigl( \sum_{j=-n_i+1}^{0}(Y_{2i-3}^{-1} Y_{2i-4}^{-1})^j + (Y_{2i-3}^{-1} Y_{2i-4}^{-1})^{-n_i}\Bigr)  & $if$ \  n_i>0\\
\displaystyle t^{-1}  Y_{2i} \sum_{j=1}^{-n_i} (Y_{2i-3}^{-1} Y_{2i-4}^{-1})^j - t^{0} \Bigl(\sum_{j=1}^{-n_i} (Y_{2i-3}^{-1} Y_{2i-4}^{-1})^j - (Y_{2i-3}^{-1} Y_{2i-4}^{-1})^{-n_i}\Bigr) & $if$ \  n_i<0\\
\end{cases} ,\\
-S'_{2i}=& \Phi \left( \frac{\partial s_{2i}}{\partial y_{2i-3}} \right)\\ =&
\begin{cases}
\displaystyle -t^{-2} Y_{2i} \sum_{j=-n_i+1}^{0} (Y_{2i-3}^{-1} Y_{2i-4}^{-1})^j Y_{2i-4} + t^{-1} \sum_{j=-n_i+1}^{0}(Y_{2i-3}^{-1} Y_{2i-4}^{-1})^j Y_{2i-4} & $if$ \  n_i>0\\
\displaystyle t^{-2} Y_{2i} \sum_{j=1}^{-n_i} (Y_{2i-3}^{-1} Y_{2i-4}^{-1})^j Y_{2i-4} - t^{-1} \sum_{j=1}^{-n_i}  (Y_{2i-3}^{-1} Y_{2i-4}^{-1})^j Y_{2i-4} & $if$ \  n_i<0\\
\end{cases} ,
\end{align*}
for $1 \leq i \leq l$.

\end{document}